\documentclass[a4paper]{article}

\usepackage[dvipdfmx]{graphicx}
\usepackage{amsmath,amssymb,amsthm}
\usepackage{mathtools}
\usepackage[matha,mathb]{mathabx}
\usepackage{thmtools}
\usepackage{hyperref}
\usepackage[capitalise]{cleveref}
\usepackage[title,titletoc]{appendix}
\usepackage[all]{xy}
\usepackage{enumitem}

\usepackage{mymacro}

\declaretheorem[numberwithin=section]{theorem}
\declaretheorem[sibling=theorem]{lemma}
\declaretheorem[sibling=theorem]{proposition}
\declaretheorem[sibling=theorem]{corollary}

\declaretheorem[style=definition,numbered=no]{definition}

\declaretheorem[style=remark,sibling=theorem]{example}
\declaretheorem[style=remark,sibling=theorem]{remark}
\declaretheorem[style=remark,numbered=no]{notation}

\numberwithin{equation}{section}
\numberwithin{figure}{section}

\title{Limits and colimits of crossed groups}
\date{\today}
\author{Jun Yoshida}

\begin{document}
\maketitle

\tableofcontents

\section*{Introduction}
\label{sec:intro}
\addcontentsline{toc}{section}{Introduction}

For a small category $\mathcal A$, a \emph{crossed $\mathcal A$-group} $G$ is a family $\{G(a)\}_{a\in\mathcal A}$ of groups equipped with
\begin{itemize}
  \item an $\mathcal A$-set structure $G(a)\times\mathcal A(b,a)\to G(b)$; and
  \item a left action $G(a)\times\mathcal A(b,a)\to\mathcal A(b,a)$
\end{itemize}
which are so compatible as if they form a double category.
The notion was originally introduced in the simplicial case $\mathcal A=\Delta$; namely, Fiedorowicz and Loday \cite{FL91} (and independently Krasauskas \cite{Kra87} with the different name \emph{skew-simplicial groups}) studied \emph{crossed simplicial groups} to understand the cyclic homology $HC_\ast(A)$ for associative algebras from a categorical viewpoint.
They pointed out that it is a special case of the homology $HG_\ast(A)$ associated with crossed simplicial groups $G$.
We here sketch the construction: first, it turns out that one can associate each crossed simplicial group $G$ with a category $\Delta_G$ called the \emph{total category} of $G$, which is endowed with a bijective-on-object functor $\Delta\to\Delta_G$.
For each natural number $n\in\mathbb N$, one also has $\operatorname{Aut}_{\Delta_G}([n])\cong G_n$.
Hence, roughly speaking, a functor $X:\Delta_G^\opposite\to\mathcal C$ is a kind of a simplicial object with \emph{$G$-symmetry}.
In particular, if $\mathcal C$ is the category $\mathbf{Mod}_k$ of modules over a commutative ring $k$, a group $HG_n(X)\in\mathbf{Mod}_k$ is 
\[
HG_n(X) := \operatorname{Tor}^{\Delta_G}_k(k,X)\ ,
\]
where the $\operatorname{Tor}^{\Delta G}_\ast$ is the $\operatorname{Tor}_\ast$ taken in the functor category $\mathbf{Mod}^{\Delta_G^\opposite}_k$.
Each algebras over $k$ gives rise to a functor $A^\otimes:\Delta_G\to\mathbf{Mod}_k$ with $A^\otimes_n\cong A^{\otimes n+1}$, so we obtain the \emph{$G$-homology group} $HG_n(A)$ for algebras.
For example, take $G=\Lambda$ the crossed simplicial set with each level set $\Lambda_n$ being the cyclic group of order $n+1$.
Then, the category $\Delta_\Lambda$ is \emph{Connes' cyclic category} \cite{Connes83} so functors $\Delta_\Lambda^\opposite\to\mathcal C$ are usually called \emph{cyclic objects} in $\mathcal C$.
The resulting homology group $HG_n(A)$ coincides with the cyclic homology group of $A$.

Note that, in \cite{FL91}, it is also pointed out that the geometric realization $|G|$ of a crossed simplicial group $G$ is canonically a topological group or sometimes even a Lie group.
Hence, crossed simplicial groups are themselves geometrically interesting.
Dyckerhoff and Kapranov \cite{DyckerhoffKapranov2015} studied them in this point of view.

Another kind of crossed groups were studied by Batanin and Markl \cite{BataninMarkl2014}; the \emph{crossed interval groups}.
We denote by $\nabla$ the category of totally ordered set with distinct maximum and minimum elements and order-preserving maps respecting them, which is sometimes called the \emph{category of intervals}.
A crossed interval group is a crossed $\nabla$-group.
They introduced the notion to investigate symmetries on the Hochschild cohomology in view of operads.
A point is the following observation: some of the structures on (symmetric) operads are controlled by the crossed interval group $\mathfrak S$ of the symmetric groups.
This evokes the idea of \emph{categories of operators} by May and Thomason \cite{MayThomason1978}, which suggests we see operads as a sort of fibrations of the form $\mathcal O^\otimes\to\mathbf{Fin}_\ast$, here $\mathbf{Fin}_\ast$ is the category of pointed finite sets and pointed maps.
Note that Lurie's definition of \emph{$\infty$-operads} \cite{Lur14} is actually based on it.
On the other hand, a planar (or non-symmetric) operad is essentially a fibration over $\nabla$; forgetting symmetries is realized as the pullback along the functor $\nabla\to\mathbf{Fin}_\ast$ collapsing $\min$ and $\max$.
Actually, the total category $\nabla_{\mathfrak S}$ lives between $\nabla$ and $\mathbf{Fin}_\ast$.
This suggests we consider operads symmetric over general crossed interval groups and classify them.

In both cases above, we can say that a crossed $\mathcal A$-group presents an intrinsic symmetry on $\mathcal A$ so that it provides a symmetric version of the notion of the ``theory $\mathcal A$;'' in the simplicial abelian case, we get the theory of $G$-symmetric homologies (cf. Dold-Kan correspondence \cite{Dold1958} between simplicial abelian groups and chain complexes) while, in the interval case, we would get the notion of $G$-symmetric operads.
Hence, the study of crossed $\mathcal A$-groups is not irrelevant to that of $\mathcal A$ itself.
In particular, the terminal crossed $\mathcal A$-group may have all the information on the intrinsic symmetry on $\mathcal A$.
Furthermore, a classification of crossed $\mathcal A$-subgroups of it will also classifies the symmetries on $\mathcal A$.
For example, according to \cite{FL91}, the terminal crossed simplicial group, say $\mathfrak W_\Delta$, consists of the hyperoctahedral groups $\mathfrak H_n:=\mathfrak S_n\ltimes(\mathbb Z/2\mathbb Z)^{\times n}$, which enables the notion of \emph{parities} as well as permutations.
The complete list of crossed simplicial subgroups of $\mathfrak W_\Delta$ is found in \cite{FL91} as well as in the appendix of this paper.

We denote by $\mathbf{CrsGrp}_{\mathcal A}$ the category of crossed interval groups, and the problem is to find the terminal object.
Unfortunately, for some reason, there were few works on the category itself, so we only have poor tools to do that.
Hence, in this paper, we aim to achieve the first step making a basic and comprehensive understanding on it, especially from a categorical viewpoint.
More precisely, we prove local presentability of $\mathbf{CrsGrp}_{\mathcal A}$ in two different ways: by the \emph{cardinality argument} and by the \emph{monadicity}.

In \cref{sec:cocompl}, we will achieve it in the first way; namely, we see there is a cardinal $\kappa$ so that only $\kappa$ many elements can involve in the operations of crossed $\mathcal A$-groups simultaneously.
In particular, it guarantees the existence of the terminal crossed $\mathcal A$-group $\mathfrak T_{\mathcal A}$.
This is a standard procedure to verify local presentability, it does not provide any constructive way of computations.
Although we give direct computations of some limits and colimits, the explicit form of $\mathfrak T_{\mathcal A}$ is still unknown at this stage.

That is why we discuss it individually in \cref{sec:terminal} where we compute $\mathfrak T_{\mathcal A}$ for a sort of small categories $\mathcal A$ satisfying certain special properties.
For example, we will obtain explicit descriptions in the simplicial and interval cases, so the computation recovers the result of \cite{FL91}.
Note that, if one obtains an explicit form of $\mathfrak T_{\mathcal A}$, it leads to a classification result of crossed $\mathcal A$-groups.
Indeed, we can index crossed $\mathcal A$-group $G$ with the image of the unique map $G\to\mathfrak T_{\mathcal A}$.
In particular, we will obtain the classification of crossed interval groups in this way, which Batanin and Markl concerned about in their paper \cite{BataninMarkl2014}.
The detail is put in Appendix \ref{sec:clsfy-crsint}, where we will prove a crossed analogue of \emph{Goursat's Lemma}.

On the other hand, in \cref{sec:monobj}, we exhibits crossed $\mathcal A$-groups as monoid objects in the slice category $\mathbf{Set}^{/\mathfrak T_{\mathcal A}}_{\mathcal A}$ over the category of $\mathcal A$-sets.
We introduce a monoidal structure on the category and prove it is monoidally closed.
As a result, the category of monoid objects, which are called \emph{crossed $\mathcal A$-monoids}, is monadic over the presheaf topos, so it is locally presentable.
It will be shown that the category of crossed $\mathcal A$-groups is both reflective and coreflective subcategory; like monoids v.s. groups.

Another important problem is the base change along functors, which we will discuss in \cref{sec:basechange}.
To avoid certain difficulties, we will not treat with arbitrary functors but only with faithful functors.
In particular, we are interested in the functors $\Delta\hookrightarrow\widetilde\Delta\to\nabla$, which adds $\max$ and $\min$ to each finite linearly ordered set.
We will see they induce adjunctions
\[
\vcenter{
  \xymatrix{
    \mathbf{CrsGrp}_\Delta \ar@/^.4pc/[]+R+(0,2);[r]+L+(0,2) \ar@/_.4pc/[]+R+(0,-2);[r]+L+(0,-2) & \mathbf{CrsGrp}_{\widetilde\Delta} \ar[l]^-\perp_-\perp & \mathbf{CrsGrp}_{\widetilde\Delta}^{/\mathfrak H\times C_2} \ar@/^.4pc/[]+R+(0,2);[r]+L+(0,2) \ar@/_.4pc/[]+R+(0,-2);[r]+L+(0,-2) & \mathbf{CrsGrp}_\nabla \ar[l]^-\perp_-\perp }}\ .
\]
using Kan extensions and \emph{Adjoint Lifting Theorem}.
This useful theorem enables us to compute left adjoint functors explicitly as soon as we have left Kan extensions.
The author expects that these basechange functors may explain how to relate symmetries available on homologies and those on monoidal categories where homologies lie in.

\subsection*{Acknowledgement}

I would like first to thank my supervisor Prof. Toshitake Kohno for the encouragement and a lot of kind support.
I would also like to mention people at Macquarie University since some important parts of the work was performed there.
Especially, I express my special gratitude to Prof. Ross Street who gave me a lot of great ideas and advice through discussions.
I would like to thank Michael Batanin and Martin Markl for the kind and polite response to my questions.
Also, I appreciate my friends Genki Sato.
He taught me many ZFC stuffs and checked my proof of local presentability.
This work was supported by the Program for Leading Graduate Schools, MEXT, Japan.
This work was supported by JSPS KAKENHI Grant Number JP15J07641.

\section{Definition and examples}
\label{sec:def}

In this first section, we recall the basic definitions and examples of crossed groups.

\begin{notation}
If $\mathcal A$ is a small category, we denote by $\mathbf{Set}_{\mathcal A}$ the category of $\mathcal A$-sets; i.e. presheaves over $\mathcal A$.
For each object $a\in\mathcal A$, we will write $\mathcal A[a]\in\mathbf{Set}_{\mathcal A}$ the presheaf represented by $a$.
\end{notation}

\begin{definition}
Let $\mathcal A$ be a small category.
Then, a \emph{crossed $\mathcal A$-group} is an $\mathcal A$-set $G\in\mathbf{Set}_{\mathcal A}$ together with data
\begin{itemize}
  \item a group structure on $G(a)$ for each $a\in\mathcal A$, and
  \item a left action $G(b)\times\mathcal A(a,b)\to\mathcal A(a,b);\ (x,\varphi)\mapsto \varphi^x$ of the group $G(b)$ on $\mathcal A(a,b)$ for each $a,b\in\mathcal A$;
\end{itemize}
which satisfy the following conditions:
\begin{enumerate}[label={\rm(CG\roman*)},leftmargin=\widthof{\indent(CG0)}]
  \item\label{condCG:gdist} for $\varphi:a\to b$ and $\psi:b\to c$ in $\mathcal A$, and for $x\in G(c)$, we have
\[
(\psi\varphi)^x = \psi^x \varphi^{\psi^\ast(x)}\ ;
\]
  \item\label{condCG:cdist} for $\varphi:a\to b$ and $x,y\in G(b)$, we have
\[
\varphi^\ast(xy) = (\varphi^y)^\ast(x)\varphi^\ast(y)\ .
\]
\end{enumerate}
\end{definition}

We can describe two conditions \ref{condCG:gdist} and \ref{condCG:cdist} more categorically: consider a map
\begin{equation}
\label{eq:def-crs}
\mathrm{crs}:G(b)\times\mathcal A(a,b)\to\mathcal A(a,b)\times G(a)
\ ;\quad (x,\varphi)\mapsto (\varphi^x,\varphi^\ast(x))\ .
\end{equation}
Then, \ref{condCG:gdist} and \ref{condCG:cdist} are respectively equivalent to the commutativities of the diagrams
\begin{equation}
\label{eq:gdist-com}
\vcenter{
  \xymatrix{
    G(c)\times\mathcal A(b,c)\times\mathcal A(a,b) \ar[d]_{\mathrm{id}\times\mathrm{comp}} \ar[r]^{\vec{\mathrm{crs}}} & \mathcal A(b,c)\times \mathcal A(a,b)\times G(a) \ar[d]^{\mathrm{comp}\times\mathrm{id}} \\
    G(c)\times\mathcal A(a,c) \ar[r]^{\mathrm{crs}} & \mathcal A(a,c)\times G(a) }}
\end{equation}
and
\begin{equation}
\label{eq:cdist-com}
\vcenter{
  \xymatrix{
    G(c)\times G(c)\times\mathcal A(b,c) \ar[d]_{\mathrm{mul}\times\mathrm{id}} \ar[r]^{\vec{\mathrm{crs}}} & \mathcal A(b,c)\times G(b)\times G(b) \ar[d]^{\mathrm{id}\times\mathrm{mul}} \\
    G(c)\times\mathcal A(b,c) \ar[r]^{\mathrm{crs}} & \mathcal A(b,c)\times G(b) }}
\end{equation}
where $\vec{\mathrm{crs}}$ indicates appropriate compositions of the map $\mathrm{crs}$.

\begin{lemma}
\label{lem:crsgrp-ptd}
Let $\mathcal A$ be a small category, and let $G$ be a crossed $\mathcal A$-group.
\begin{enumerate}[label={\rm(\arabic*)}]
  \item\label{req:crsgrp-ptd:id} For each $a\in\mathcal A$, the action of $G(a)$ on $\mathcal A(a,a)$ preserves the identity morphism; i.e. $\mathrm{id}^x=\mathrm{id}$ for any $x\in G(a)$.
  \item\label{req:crsgrp-ptd:mono}
For each $a,b\in\mathcal A$, the action of $G(b)$ on $\mathcal A(a,b)$ preserves monomorphisms and split epimorphisms.
  \item\label{req:crsgrp-ptd:unit} For every morphism $\varphi:a\to b\in\mathcal A$, the map $\varphi^\ast:G(b)\to G(a)$ preserves the units of the groups.
Moreover, if $\varphi$ is $G(b)$-invariant, then $\varphi^\ast$ is a group homomorphism.
\end{enumerate}
\end{lemma}
\begin{proof}
The assertions \ref{req:crsgrp-ptd:id} and \ref{req:crsgrp-ptd:unit} easily follow from the conditions \ref{condCG:gdist} and \ref{condCG:cdist} respectively.
It remains to show \ref{req:crsgrp-ptd:mono}.
It immediately follows from \ref{condCG:gdist} that the action of $G(a)$ on $\mathcal A(a,b)$ preserves split epimorphisms.
To see it also preserves monomorphisms, take an arbitrary monomorphism $\delta:a\to b$ in $\mathcal A$, and let $x\in G(b)$.
Given two morphisms $\varphi_1,\varphi_2:c\to a$, suppose $\delta^x\varphi_1=\delta^x\varphi_2$.
By the condition \ref{condCG:gdist}, we have
\[
\delta^x\varphi_i
=\delta^x(\varphi_i^{\delta^\ast(x)^{-1}})^{\delta^\ast(x)}
=(\delta\varphi_i^{\delta^\ast(x)^{-1}})^x
\]
for $i=1,2$.
Since $\delta$ is a monomorphism, we get $\varphi_1=\varphi_2$.
This implies $\delta^x$ is a monomorphism, and we obtain \ref{req:crsgrp-ptd:mono}.
\end{proof}

\begin{corollary}
\label{cor:crsgrp-term}
Let $\mathcal A$ be a small category.
Then, if $t\in\mathcal A$ is a terminal object, for every crossed $\mathcal A$-group $G$, and for each object $a\in\mathcal A$, the unique map $a\to t$ induces a group homomorphism $G(t)\to G(a)$.
Dually, if $s\in\mathcal A$ is an initial object, the unique map $s\to a$ induces a group homomorphism $G(a)\to G(s)$.
\end{corollary}

\begin{corollary}
\label{cor:noncrs-grp}
Let $\mathcal A$ be a small category, and let $G$ be a crossed $\mathcal A$-group.
Then, if the action $G(a)$ on $\mathcal A(b,a)$ is trivial for each $a,b\in\mathcal A$, then $G$ is $\mathcal A$-group; i.e. a group object in the category $\mathbf{Set}_{\mathcal A}$.
\end{corollary}

\begin{remark}
The converse of \cref{cor:noncrs-grp} holds: every $\mathcal A$-group can be seen as a crossed $\mathcal A$-group with the trivial actions on each $\mathcal A(a,b)$.
To distinguish such a crossed group from general crossed groups, we often say it is \emph{non-crossed}.
If we write $\mathbf{Grp}_{\mathcal A}$ the category of (non-crossed) $\mathcal A$-groups, there is a fully faithful embedding
\[
\mathbf{Grp}_{\mathcal A}\to\mathbf{CrsGrp}_{\mathcal A}\ .
\]
\end{remark}

\begin{definition}
Let $\mathcal A$ be a small category, and let $G$ and $H$ be crossed $\mathcal A$-groups.
Then, a map $G\to H$ of crossed $\mathcal A$-groups is a map of $\mathcal A$-sets which is a degreewise group homomorphism respecting the actions on $\mathcal A(a,b)$ for each $a,b\in\mathcal A$.
\end{definition}

Clearly, crossed $\mathcal A$-groups and maps of them form a category, which we will denote by $\mathbf{CrsGrp}_{\mathcal A}$.
Then, the following result is obvious.

\begin{proposition}
\label{prop:crsgrp-initial}
For every small category $\mathcal A$, the category $\mathbf{CrsGrp}_{\mathcal A}$ admits an initial object; namely, the terminal $\mathcal A$-set $\ast$ with the unique crossed group structure.
\end{proposition}

Because of the compatibility condition in the definition of maps of crossed groups, the terminal $\mathcal A$-set $\ast$ is no longer terminal in the category $\mathbf{CrsGrp}_{\mathcal A}$ in general.
Nevertheless, for each crossed $\mathcal A$-group $G$, there still exists the unique $\mathcal A$-set map $G\to\ast$, and it makes sense to ask whether it is a map of crossed groups or not.
In this point of view, we can rephrase \cref{cor:noncrs-grp} as follows.

\begin{corollary}
\label{cor:noncrs-triv}
Let $\mathcal A$ be a small category.
Then, a crossed $\mathcal A$-group $G$ is a (non-crossed) $\mathcal A$-group if and only if the unique $\mathcal A$-map $G\to\ast$ is a map of crossed $\mathcal A$-groups.
Consequently, there is a canonical equivalence
\[
\mathbf{Grp}_{\mathcal A}
\simeq \mathbf{CrsGrp}^{/\ast}_{\mathcal A}
\]
of categories which respects the canonical functors into $\mathbf{CrsGrp}_{\mathcal A}$.
\end{corollary}

In \cref{sec:terminal}, we will see the category $\mathbf{CrsGrp}_{\mathcal A}$ also admits a terminal object, which is hard to compute in general.
In particular, it does not necessarily conincides with the initial object $\ast\in\mathbf{CrsGrp}_{\mathcal A}$, and this implies the category $\mathbf{CrsGrp}_{\mathcal A}$ is not pointed.
Nevertheless, we can consider the notion of \emph{images and kernels} of maps of crossed groups in an intuitive way.
Indeed, if $\varphi:G\to H$ is a map of crossed $\mathcal A$-groups, then it turns out that its degreewise image and kernel respectively form crossed $\mathcal A$-groups.
An important consequence of \cref{cor:noncrs-triv} is that the kernel of a map of crossed groups is always a (non-crosssed) $\mathcal A$-group.

In the case $\mathcal A$ is the simplex category $\Delta$, crossed $\Delta$-groups are usually called \emph{crossed simplicial groups}, and they were of central interest in previous works such as \cite{Kra87}, \cite{FL91}, and \cite{DyckerhoffKapranov2014}.
We recall the definition of $\Delta$; it is the category whose objects are the totally ordered sets
\[
[n] := \{0,\dots,n\}
\]
for $n\in\mathbb N$ and whose morphisms are order-preserving maps.
We, however, here introduce two other categories, which are more important than $\Delta$ throughout this paper.
Define a category $\widetilde\Delta$ to consists of the totally ordered sets
\[
\langle n\rangle := \{1,\dots,n\}
\]
for $n\in\mathbb N$ as objects and order-preserving maps as morphisms.
The third category $\nabla$ looks also similar but of a bit technically different form; the objects are totally ordered sets
\[
\double\langle n\double\rangle := \{-\infty,1,\dots,n,\infty\}
\]
for $n\in\mathbb N$, and the morphisms are order-preserving maps that respects $-\infty$ and $\infty$.
We call crossed $\widetilde\Delta$-groups and crossed $\nabla$-groups \emph{augmented crossed simplicial groups} and \emph{crossed interval groups} respectively; the latter is due to \cite{BataninMarkl2014}.
We have a canonical functors
\begin{equation}
\label{eq:deltainterval}
\xymatrix@R=0pt{
  \Delta \ar[r] & \widetilde\Delta \ar[r] & \nabla \\
  [k] \ar@{|->}[r] & \langle k+1\rangle,\ \langle l\rangle \ar@{|->}[r] & \double\langle l\double\rangle }\ .
\end{equation}
The first functor is fully faithful, and the last is faithful and bijective on objects.

\begin{remark}
The category $\nabla$ is sometimes called the \emph{category of intervals}.
As pointed out in the unpublished note \cite{Joyal97} by Joyal, we have an isomorphism
\[
\nabla\cong\Delta^\opposite\ .
\]
Although we do not use this fact in this paper, the reader should notice that the Segal condition on simplicial spaces can also be seen as a condition on functors $\nabla\to\mathbf{Spaces}$.
\end{remark}

Now, we are giving several examples below.
For this, it is convenient to use \emph{joins of ordered sets}: for two partially ordered sets $P$ and $Q$, we denote by $P\star Q$ their join.
For example, there is a unique isomorphism $\langle m\rangle\star\langle n\rangle\cong\langle m+n\rangle$ of ordered sets.
Note that the underlying set of $P\star Q$ is just $P\amalg Q$.

\begin{example}
\label{ex:sym-crs}
We define a crossed interval group $\mathfrak S$ as follows:
\begin{itemize}
  \item for each $n\in\mathbb N$, $\mathfrak S_n$ is the $n$-th permutation group, or the permutation group on the set $\double\langle n\double\rangle$ fixing $\pm\infty$;
  \item for $\varphi:\double\langle m\double\rangle\to\double\langle n\double\rangle\in\nabla$, we define $\varphi^\ast:\mathfrak S_n\to\mathfrak S_m$ as follows: for $\sigma\in\mathfrak S_n$, $\varphi^\ast(\sigma)$ is the permutation on $\double\langle m\double\rangle$ given as the composition
\[
\begin{multlined}
\langle m\rangle
\cong\varphi^{-1}\{-\infty\}\star\varphi^{-1}\{1\}\star\dots\star\varphi^{-1}\{n\}\star\varphi^{-1}\{\infty\} \\
\to\varphi^{-1}\{-\infty\}\star\varphi^{-1}\{\sigma^{-1}(1)\}\star\dots\star\varphi^{-1}\{\sigma^{-1}(n)\}\star\varphi^{-1}\{\infty\}
\cong\langle m\rangle\mathrlap{\ ,}
\end{multlined}
\]
where the first and the last maps are the unique order-preserving bijections;
  \item for $\sigma\in\mathfrak S_n$, the action on $\nabla(\double\langle m\double\rangle,\double\langle n\double\rangle)$ is given by
\[
\begin{multlined}
\varphi^\sigma:\double\langle m\double\rangle
\cong \varphi^{-1}\{-\infty\}\star\varphi^{-1}\{\sigma^{-1}(1)\}\star\dots\star\varphi^{-1}\{\sigma^{-1}(n)\}\star\varphi^{-1}\{\infty\}\\
\to \{-\infty\}\star\{1\}\dots\{n\}\star\{\infty\}
\cong \double\langle n\double\rangle\mathrlap{\ .}
\end{multlined}
\]
\end{itemize}
The conditions \ref{condCG:gdist} and \ref{condCG:cdist} are verified easily.
We also have similar constructions for braid groups, pure braid groups, and so on.
\end{example}

\begin{example}
\label{ex:cyc-crs}
For each natural number $n$, we denote by $C_n$ the cyclic group of order $n$, which is canonically embedded in the symmetric group $\mathfrak S_n$.
Although the subsets $C_n\subset\mathfrak S_n$ do not define an interval subset of $\mathfrak S$ defined in \cref{ex:sym-crs}, one can see that they actually form an augmented simplicial subset, which we denote by $\Lambda$, so $\Lambda_n=C_n$.
The augmented simplicial structure is explicitly described as follows: for $\mu:\langle m\rangle\to\langle n\rangle\in\widetilde\Delta$, $\mu^\ast:\Lambda_n\to\Lambda_m$ is given by
\[
\mu^\ast(\sigma)(i) \equiv i+\sigma(\mu(i))-\mu(i) \pmod n\ .
\]
Clearly $\Lambda$ is a crossed augmented simplicial group.
Note that the category $\Delta_\Lambda$ defined in \cref{prop:crsgrp-cat} is called the \emph{Connes' cycle category} after Connes' work \cite{Connes83} on the cyclic homology and sometimes denoted just by $\Lambda$ (see \cref{rem:crsgrp-cat-id}).
\end{example}

\begin{example}
\label{ex:crsgrp-intHypoct}
Recall that the wreath product of a group $G$ by $\mathfrak S_n$ is the group
\[
G\wr\mathfrak S_n := \mathfrak S_n\ltimes G^{\times n}
\]
whose underlying set is $\mathfrak S_n\times G^{\times n}$ with multiplication given by
\[
(\sigma;\vec x)(\tau;\vec y)
= (\sigma\tau;\tau^\ast(\vec x)\cdot \vec y)\ ,
\]
where, if $\vec x=(x_1,\dots,x_n)$ and $\vec y=(y_1,\dots,y_n)$
\[
\tau^\ast(\vec x)\cdot\vec y
:= (x_{\tau(1)}y_1,\dots,x_{\tau(n)}y_n)\ .
\]
In particular the case $G=C_2$ is the cyclic group of order $2$, $H_n:=C_2\wr\mathfrak S_n$ is the Weyl group of the root system $B_n$, which is called the \emph{$n$-th hyperoctahedral group}.

We claim that the family $\{H_n\}_n$ forms a crossed interval group $\mathfrak H$.
Set $\mathfrak H(\double\langle n\double\rangle):=H_n$, and for $\varphi:\double\langle m\double\rangle\to\double\langle n\double\rangle\in\nabla$, we define $\varphi^\ast:H_n\to H_m$ by
\[
\varphi^\ast(\sigma;\vec\varepsilon)
:= (\varphi^\ast(\sigma)\beta_\varphi(\varepsilon);\varphi^\ast(\vec\varepsilon))\ ,
\]
where
\begin{itemize}
  \item $\varphi^\ast(\sigma)$ is the permutation on $\double\langle m\double\rangle$ defined in \cref{ex:sym-crs};
  \item $\beta_\varphi(\vec\varepsilon)$ is the permutation given by
\[
\begin{multlined}
\double\langle m\double\rangle
\cong \varphi^{-1}\{-\infty\}\star\varphi^{-1}\{1\}\star\dots\star\varphi^{-1}\{n\}\star\varphi^{-1}\{\infty\} \\
\xrightarrow{\mathrm{id}\amalg\beta^{\varepsilon_1} \amalg\dots\amalg \beta^{\varepsilon_n}\amalg\mathrm{id}} \varphi^{-1}\{-\infty\}\star\varphi^{-1}\{1\}\star\dots\star\varphi^{-1}\{n\}\star\varphi^{-1}\{\infty\}\cong \double\langle m\double\rangle\mathrlap{\ ,}
\end{multlined}
\]
where each $\beta:\varphi^{-1}\{j\}\to\varphi^{-1}\{j\}$ is the order-reversing map;
  \item if $\vec\varepsilon=(\varepsilon_1,\dots,\varepsilon_n)$, then
\[
\varphi^\ast(\vec\varepsilon)=(\varepsilon_{\varphi(1)},\dots,\varepsilon_{\varphi(m)})
\]
with the assumption $\varepsilon_{\pm\infty}=1$.
\end{itemize}
This actually defines an interval set
\[
\mathfrak H:\nabla^\opposite\to\mathbf{Set}\ .
\]
It is moreover verified that $\mathfrak H$ together with the action of $H_n$ on $\nabla(\double\langle m\double\rangle,\double\langle n\double\rangle)$ through $H_n\to\mathfrak S_n$ is a crossed interval group, which is called \emph{Hyperoctahedral crossed interval group}.
\end{example}

\begin{example}
\label{ex:crsgrp-simpWeyl}
Let $G$ be a crossed $\mathcal A$-group.
If $\varphi:\mathcal A'\to\mathcal A$ is a faithful functor such that the image of the map
\[
\varphi:\mathcal A'(a,b)\to\mathcal A(F(a),F(b))
\]
is stable under $G(\varphi(b))$-action, then the restricted $\mathcal A$-set $\varphi^\ast(G)$ inherits a structure of crossed $\mathcal A'$-groups.
For example, it is verified that the canonical functors $\Delta\hookrightarrow\widetilde\Delta\hookrightarrow\nabla$ given in \eqref{eq:deltainterval} pull back the Hyperoctahedral crossed interval group $\mathfrak H$ to crossed groups on $\Delta$ and $\widetilde\Delta$ respectively.
We call both of them \emph{Hyperoctahedral crossed simplicial groups} and also denote by $\mathfrak H$ by abuse of notation.
This construction is discussed in more detail in \cref{sec:basechange}.
\end{example}

We finally mention that we can associate each crossed $\mathcal A$-group with a category that is an extension of $\mathcal A$.
More precisely, we have the following result, which is a direct consequence of the commutative squares \eqref{eq:gdist-com} and \eqref{eq:cdist-com} together with \cref{lem:crsgrp-ptd}.

\begin{proposition}
\label{prop:crsgrp-cat}
Let $\mathcal A$ be a small category.
Suppose we are given an $\mathcal A$-set $G$ together with a group structure on $G(a)$ and a left action on $\mathcal A(b,a)$ for each $a,b\in\mathcal A$.
Then, $G$ is a crossed $\mathcal A$-group if and only if the following data defines a category $\mathcal A_G$:
\begin{itemize}
  \item\textbf{object}: the same as $\mathcal A$;
  \item\textbf{morphism}: for $a,b\in\mathcal A$, $\mathcal A_G(a,b)=\mathcal A(a,b)\times G(a)$;
  \item\textbf{composition}: given by
\[
\begin{split}
\mathcal A(b,c)\times G(b)\times\mathcal A(a,b)\times G(a)
&\xrightarrow{\mathrm{id}\times\mathrm{crs}\times\mathrm{id}}\mathcal A(b,c)\times\mathcal A(a,b)\times G(a)\times G(a) \\
&\xrightarrow{\mathrm{comp}\times\mathrm{mul}}\mathcal A(a,c)\times G(a)\ ,
\end{split}
\]
\end{itemize}
where the map $\operatorname{crs}$ is one defined by \eqref{eq:def-crs}.
In this case, the canonical map $\mathcal A(a,b)\to\mathcal A_G(a,b)$ defines a functor which is faithful and bijective on objects.
\end{proposition}

\begin{remark}
\label{rem:crsgrp-cat-id}
If $\mathcal A$ has no non-trivial isomorphism, then for each crossed $\mathcal A$-group $G$, and for each $a\in\mathcal A$, we have
\[
G(a)=\operatorname{Aut}_{\mathcal A_G}(a)\ .
\]
Moreover, the whole crossed $\mathcal A$-group structure on $G$ can be recovered as follows: for $x\in G(a)=\operatorname{Aut}_{\mathcal A_G}(a)$ and for $f:b\to a\in\mathcal A$, the composition $xf:b\to a\in\mathcal A_G$ is uniquely represented by a pair $(g,y)$ with $g:b\to a\in\mathcal A$ and $y\in G(b)$.
Clearly, $g=f^x$ and $y=f^\ast(x)$.
In other words, the crossed $\mathcal A$-group structure on $G$ involves the unique factorization property of the category $\mathcal A_G$.
In the case $\mathcal A$ is the simplex category $\Delta$, formal statements and proofs will be found in \cite{FL91} (Proposition 1.7).
Note that, in these papers, crossed simplicial groups are defined as extensions of the category $\Delta$.
\end{remark}

\section{Cocompleteness and completeness}
\label{sec:cocompl}

We investigate elementary properties of the category $\mathbf{CrsGrp}_{\mathcal A}$ with regard to colimits and limits.
Throughout the section, we fix a small category $\mathcal A$.
Note that the problem is not as easy as in the case of usual algebraic categories over $\mathcal A$.
For example, as pointed out in the previous section, terminal objects in $\mathbf{CrsGrp}_{\mathcal A}$ are already highly non-trivial, and it shows that the forgetful functor $\mathbf{CrsGrp}_{\mathcal A}\to\mathbf{Set}_{\mathcal A}$ does not preserve limits.

We begin with colimits since the situation is somehow easier than limits.
For a small category $\mathcal A$, we denote by $\mathcal A_0$ the maximal discrete subcategory of $\mathcal A$; that is, the subcategory with the same objects and the identities.
Since crossed $\mathcal A$-groups are by definition degreewise groups, we have the forgetful functor $\mathbf{CrsGrp}_{\mathcal A}\to\mathbf{Grp}_{\mathcal A_0}$.

\begin{proposition}
\label{prop:crsgrp-colim}
The forgetful functor $\mathbf{CrsGrp}_{\mathcal A}\to\mathbf{Grp}_{\mathcal A_0}$ creates arbitrary small colimits.
Consequently, the category $\mathbf{CrsGrp}_{\mathcal A}$ is cocomplete.
\end{proposition}
\begin{proof}
Thanks to \cref{prop:crsgrp-initial}, $\mathbf{CrsGrp}_{\mathcal A}$ has an initial object which is still initial in $\mathbf{Grp}_{\mathcal A_0}$.
Thus, it suffices to show the functor $\mathbf{CrsGrp}_{\mathcal A}\to\mathbf{Grp}_{\mathcal A_0}$ creates both pushouts and filtered colimits.
The latter is obvious since only finitely many variables appear simultaneously in the axioms of crossed $\mathcal A$-groups.

As for pushouts, suppose we are given a span
\[
G_1\xleftarrow{f_1} H \xrightarrow{f_2} G_2
\]
in the category $\mathbf{CrsGrp}_{\mathcal A}$.
It is well-known that the pushout of the span in $\mathbf{Grp}_{\mathcal A_0}$, which we write $G_1\ast_H G_2$ to distinguish it with the pushout of the underlying $\mathcal A$-sets, is given as follows: for each $a\in\mathcal A_0$, the group $(G_1\ast_H G_2)(a)$ is the quotient of the free monoid over the set $G_1\cup G_2$, regarding $G_1\cap G_2=\{e\}$, by the relation $\sim$ generated by
\begin{itemize}
  \item insertions and deletions of the unit, i.e. for $x_i\in G_1\cup G_2$,
\[
(x_1,\dots,x_n)
\sim (x_1,\dots,x_k,e,x_{k+1},\dots,x_n)\ ;
\]
  \item multiplicativities of $G_1$ and $G_2$, i.e. for $x_i\in G_1\cup G_2$ with $x_k$ and $x_{k+1}$ lying in the common group,
\[
(x_1,\dots,x_n)
\sim (x_1,\dots,x_{k-1},x_kx_{k+1},x_{k+2},\dots,x_n)\ ;
\]
  \item $H$-invariance, i.e. for $x_i\in G_1\cup G_2$ with $x_k\in G_j$ and $x_{k+1}\in G_{j'}$, and for $h\in H$,
\[
(x_1,\dots,x_n)
\sim (x_1,\dots,x_{k-1},x_k f_j(h)^{-1},f_{j'}(h)x_{k+1},x_{k+2},\dots,x_n)\ .
\]
\end{itemize}
For each morphism $\varphi:a\to b\in\mathcal A$, we define a map $\varphi^\ast:(G_1\ast_H G_2)(b)\to (G_1\ast_H G_2)(a)$ inductively by
\[
\begin{gathered}
\varphi^\ast(e)=e \\
\varphi^\ast(x_1,\dots,x_n)=(\varphi^{x_n})^\ast(x_1,\dots,x_{n-1})\varphi^\ast(x_n)\ .
\end{gathered}
\]
It is easily verified the definition is invariant under the relation above, so $\varphi^\ast$ is well-defined.
Now, it is tedious but not difficult to see it is a unique structure which makes $G_1\ast_H G_2$ into a crossed $\mathcal A$-group, which completes the proof.
\end{proof}

\begin{corollary}
\label{cor:crsgrp-filtcolim}
The forgetful functor $\mathbf{CrsGrp}_{\mathcal A}\to\mathbf{Set}_{\mathcal A}$ creates filtered colimits.
\end{corollary}
\begin{proof}
We have the following commutative square of forgetful functors:
\[
\xymatrix{
  \mathbf{CrsGrp}_{\mathcal A} \ar[r] \ar[d] & \mathbf{Set}_{\mathcal A} \ar[d] \\
  \mathbf{Grp}_{\mathcal A_0} \ar[r] & \mathbf{Set}_{\mathcal A_0} }
\]
By \cref{prop:crsgrp-colim}, the left functor creates filtered colimits, and it is well-known that so does the right and the bottom.
Hence, the top functor also creates filtered colimits, which is exactly what we want to show.
\end{proof}

We now get into a difficult part: the limits.
Fortunately, it turns out that the essential difficulty comes only from terminal objects and not from pullbacks.

\begin{proposition}
\label{prop:crsgrp-limconn}
Let $C$ be a category with a terminal object.
Then, the forgetful functor $\mathbf{CrsGrp}_{\mathcal A}\to\mathbf{Set}_{\mathcal A}$ creates $C$-limits.
In particular, the category $\mathbf{CrsGrp}_{\mathcal A}$ has pullbacks which are computed degreewisely.
\end{proposition}
\begin{proof}
Put $t\in C$ a terminal object, and suppose we are given a functor $G_\bullet:C\to\mathbf{CrsGrp}_{\mathcal A}$.
Taking the limit in the category $\mathbf{Set}_{\mathcal A}$, we put $G_\infty=\lim_CG_\bullet\in\mathbf{Set}_{\mathcal A}$.
Note that $G_\infty$ admits a unique degreewise group structure so that the canonical $\mathcal A$-map $G_\infty\to G_c$ is a degreewise group homomorphism for each $c\in C$.
In particular, for each $a\in\mathcal A$, the group $G_\infty(a)$ inherits an action on $\mathcal A(b,a)$ for each $b\in\mathcal A$ through the homomorphism $G_\infty(a)\to G_t(a)$.
One can then see these structures exhibit $G_\infty$ as a crossed $\mathcal A$-group.
\end{proof}

In order to show the existence of a terminal object, we furthermore show the category $\mathbf{CrsGrp}_{\mathcal A}$ is locally presentable.
Indeed, it is known that as for locally presentable categories, limits can be realized as \emph{colimits}.
Fortunately, $\mathbf{CrsGrp}_{\mathcal A}$ is cocomplete, so, in view of Corollary 2.47 in \cite{AdamekRosicky1994}, we only have to show the accessibility.
Now, for a small category $C$, we denote by $|C|$ the cardinality of the set of morphisms of $C$.
Then, in the rest of the section, we are to prove the following theorem.

\begin{theorem}
\label{theo:crsgrp-locpres}
Let $\mathcal A$ be a small category, and let $\kappa$ be an (infinite) regular cardinal greater than $\omega\times|\mathcal A|$, here $\omega$ is the smallest infinite cardinal.
Then, the category $\mathbf{CrsGrp}_{\mathcal A}$ is locally $\kappa$-presentable; i.e. it is cocomplete, and every object can be written as a $\kappa$-filtered colimit of $\kappa$-small objects.
\end{theorem}

\begin{corollary}
\label{cor:crsgrp-compl}
For every small category $\mathcal A$, the category $\mathbf{CrsGrp}_{\mathcal A}$ is both complete and cocomplete.
In particular, it admits a terminal object.
\end{corollary}

We first give a criterion for the smallness of object in $\mathbf{CrsGrp}_{\mathcal A}$.
Recall that, as for an ordinary set $S$, it is $\kappa$-small in $\mathbf{Set}$ precisely if $|S|<\kappa$.
Thus, a key idea is considering the functor
\[
\underline{(\blankdot)}:\mathbf{CrsGrp}_{\mathcal A}\to\mathbf{Set}
\ ;\quad G\mapsto\underline G:= \coprod_{a\in\mathcal A}G(a)
\]
to compare smallnesses in these two categories.

\begin{lemma}
\label{lem:crsgrp-smallobj}
Let $\kappa$ be a regular cardinal greater than $|\mathcal A|$.
Then, a crossed $\mathcal A$-group $G$ is $\kappa$-small provided the cardinality of the set $\underline G$ is less than $\kappa$.
\end{lemma}
\begin{proof}
Let $G\in\mathbf{CrsGrp}_{\mathcal A}$ with $|\underline G|<\kappa$.
We have to show that the functor
\[
\mathbf{CrsGrp}_{\mathcal A}(G,\blankdot):\mathbf{CrsGrp}_{\mathcal A}\to\mathbf{Set}
\]
preserves $\kappa$-filtered colimits.
In view of Corollary 1.7 in \cite{AdamekRosicky1994}, it suffices to verify the preservation only for sequential ones.
Suppose $\lambda$ is an ordinal of cofinality at least $\kappa$; i.e. every subset $S\subset\lambda$ with $|S|<\kappa$ has a supremum $\sup S\in\lambda$, and take a diagram $H_\bullet:\lambda\to\mathbf{CrsGrp}_{\mathcal A}$.
Writing $H_\infty:=\colim H_\bullet$ for simplicity, we show the canonical map
\begin{equation}
\label{eq:prf:crsgrp-smallobj:canmap}
\colim_{\alpha<\lambda}\mathbf{CrsGrp}_{\mathcal A}(G,H_\alpha)\to\mathbf{CrsGrp}_{\mathcal A}(G,H_\infty)
\end{equation}
is bijective.

Note that, by virtue of \cref{cor:crsgrp-filtcolim}, there is a canonical identification $\underline{H_\infty}\cong \colim_{\alpha<\lambda} \underline{H_\alpha}$, so we obtain a commutative square
\begin{equation}
\label{eq:prf:crsgrp-smallobj:}
\vcenter{
  \xymatrix{
    \colim_{\alpha<\lambda}\mathbf{CrsGrp}_{\mathcal A}(G,H_\alpha) \ar[r] \ar[d]_{\underline{(\blankdot)}} & \mathbf{CrsGrp}_{\mathcal A}(G,H_\infty) \ar[d]^{\underline{(\blankdot)}} \\
    \colim_{\alpha<\lambda}\mathbf{Set}(\underline G,\underline{H_\alpha}) \ar[r]^\simeq & \mathbf{Set}(\underline G,\colim_{\alpha<\lambda}\underline{H_\alpha}) }}
\end{equation}
In view of the criterion of the smallness of sets, the bottom map is bijecttive.
On the other hand, since the functor $\underline{(\blankdot)}$ is faithful, and since filtered colimits in $\mathbf{Set}$ preserves injections, the vertical maps are injective.
Hence, it immediately follows the map \eqref{eq:prf:crsgrp-smallobj:canmap} is injective.
Moreover, for each map $f:G\to H_\infty$ of crossed $\mathcal A$-group, the underlying map $\underline f:\underline G\to\underline{H_\infty}$ factors through a map $\bar f:\underline G\to\underline{H_{\alpha'_0}}$ followed by the structure map $\underline{H_{\alpha'_0}}\to\underline{H_\infty}$ for some ordinal $\alpha'_0<\lambda$.
This does not imply $\bar f$ is an underlying map of a map of crossed $\mathcal A$-group.
Nevertheless, there are functions
\[
\beta:\coprod_{a,b\in\mathcal A}\mathcal A(b,a)\times G(a)\to \lambda
\ ,\quad
\gamma:\coprod_{a\in\mathcal A}G(a)\times G(a)\to \lambda\ ,
\]
such that
\begin{enumerate}[label={\rm(\roman*)}]
  \item $\beta(\varphi,x),\gamma(x,y)>\alpha'_0$ for each $\varphi\in\mathcal A(a,b)$ and $x,y\in G(b)$;
  \item the map $\underline{H_{\alpha'_0}}\to\underline{H_{\beta(\varphi,x)}}$ identifies the elements $\bar f(\varphi^\ast(x))$ and $\varphi^\ast(\bar f(x))$;
  \item the map $\underline{H_{\alpha'_0}}\to\underline{H_{\gamma(x,y)}}$ identifies the elements $\bar f(xy^{-1})$ with $\bar f(x)\cdot\bar f(y)^{-1}$.
\end{enumerate}
The set $\{\beta(\varphi,x)\mid\varphi,x\}\cup\{\gamma(x,y)\mid x,y\}$ is of cardinality $|\mathcal A|\times(|\mathcal A|+|\underline G|)<\kappa$, so the cofinality of $\lambda$ implies there is an ordinal $\alpha_0<\lambda$ with $\alpha_0>\beta(\varphi,x),\gamma(x,y)$ for every $\varphi$, $x$, and $y$.
Now, it is easily verified that the composition
\[
\underline G
\xrightarrow{\bar f}\underline{H_{\alpha'_0}}
\to \underline{H_{\alpha_0}}
\]
underlies a map $f_0:G\to H_{\alpha_0}$ of crossed $\mathcal A$-group, and the map $f:G\to H_\infty$ factors through $f_0$.
In other words, $f$ is the image of $f_0\in\mathbf{CrsGrp}_{\mathcal A}(G,H_{\alpha_0})$.
This implies that the map \eqref{eq:prf:crsgrp-smallobj:canmap} is also surjective.
\end{proof}

Fix a regular cardinal $\kappa$ as in \cref{theo:crsgrp-locpres}, and define $\mathbf{CrsGrp}_{\mathcal A}^{<\kappa}$ to be the full subcategory of $\mathbf{CrsGrp}_{\mathcal A}$ spanned by crossed $\mathcal A$-groups $G$ with $|\underline G|<\kappa$.
According to \cref{lem:crsgrp-smallobj}, all objects in $\mathbf{CrsGrp}_{\mathcal A}^{<\kappa}$ is $\kappa$-small in $\mathbf{CrsGrp}_{\mathcal A}$.

\begin{lemma}
\label{lem:crsgrp-kappa-s}
The category $\mathbf{CrsGrp}_{\mathcal A}^{<\kappa}$ is essentially small.
\end{lemma}
\begin{proof}
Since the category $\mathbf{CrsGrp}_{\mathcal A}$ is locally small, it suffices to show it has only finitely many isomorphism classes.
Notice that, for an indexed family $\{X(a)\}_{a\in\mathcal A}$ of sets with $|X(a)|<\kappa$, a structure of crossed $\mathcal A$-groups can be seen an element of the set
\[
\begin{multlined}
\prod_{\varphi:a\to b\in\mathcal A}\mathbf{Set}(X(b),X(a))
\times\prod_{a\in\mathcal A}\mathbf{Set}(X(a)\times X(a),X(a)) \\
\times\prod_{a,b\in\mathcal A}\mathbf{Set}(X(a),\operatorname{Aut}_{\mathbf{Set}}(\mathcal A(b,a)))
\end{multlined}
\]
whose cardinality is bounded above by
\[
(\kappa^\kappa)^{|\mathcal A|}\times(\kappa^{\kappa\times\kappa})^{|\mathcal A|}\times((|\mathcal A|^{|\mathcal A|})^\kappa)^{|\mathcal A|\times|\mathcal A|}
=\kappa^\kappa
\]
Then, together with the cardinality of choices of the family $\{X(a)\}_{a\in\mathcal A}$, one can see there are only at most $\kappa^{|\mathcal A|}\times\kappa^\kappa=\kappa^\kappa$ isomorphism classes.
\end{proof}

\begin{lemma}
\label{lem:crsgrp-kappa-cocompl}
The subcategory of $\mathbf{CrsGrp}_{\mathcal A}^{<\kappa}\subset\mathbf{CrsGrp}_{\mathcal A}$ is closed under $\kappa$-small colimits; i.e. for every functor $C\to\mathbf{CrsGrp}_{\mathcal A}^{<\kappa}$, its colimit in $\mathbf{CrsGrp}_{\mathcal A}$ again belongs to $\mathbf{CrsGrp}_{\mathcal A}^{<\kappa}$
\end{lemma}
\begin{proof}
Let $G_\bullet:C\to\mathbf{CrsGrp}_{\mathcal A}^{<\kappa}$ be a diagram with $|C|<\kappa$.
Since the category $\mathbf{CrsGrp}_{\mathcal A}$ is cocomplete by \cref{prop:crsgrp-colim}, we put $G_\infty:=\colim_{c\in C} G_c\in\mathbf{CrsGrp}_{\mathcal A}$.
According to \cref{prop:crsgrp-colim}, for each $a\in\mathcal A$, there is a surjection
\[
\left(\coprod_{c\in C}G_c(a)\right)^\ast
\twoheadrightarrow G_\infty(a)\ ,
\]
where $(\blankdot)^\ast:\mathbf{Set}\to\mathbf{Set}$ is the \emph{free monoid monad}, while we have
\[
\left|\left(\coprod_{c\in C}G_c(a)\right)^\ast\right|
\le \omega\times\sum_{c\in C} |G_c(a)|
\le \omega\times|C|\times|\underline G|
\]
for each $a\in\mathcal A$.
It follows that
\[
|\underline{G_\infty}|
\le |\mathcal A|\times \omega\times|C|\times|\underline G|
< \kappa\ ,
\]
which shows $G_\infty\in\mathbf{CrsGrp}_{\mathcal A}^{<\kappa}$.
Thus, the required result follows.
\end{proof}

\begin{lemma}
\label{lem:crsgrp-subgen}
Let $G$ be a crossed $\mathcal A$-group.
Suppose we are given a subset $S\subset\underline G$ of cardinality less than $\kappa$.
Then, there is a crossed $\mathcal A$-subgroup $G'\subset G$ such that $S\subset\underline G'$ and $G'\in\mathbf{CrsGrp}_{\mathcal A}^{<\kappa}$.
\end{lemma}
\begin{proof}
If $S=\varnothing$, the statement is obvious.
In the case $S$ is a singleton, say $S=\{x\}$ with $x\in G(a_0)$, for each $a\in\mathcal A$, we set $G'(a)\subset G(a)$ to be the subgroup generated by the subset
\begin{equation}
\label{eq:prf:crsgrp-subgen:single}
\left\{\varphi^\ast(x^\varepsilon)\ \middle|\ \varepsilon=\pm1,\ \varphi:a\to a_0\in\mathcal A\right\}\ .
\end{equation}
Since $G'(a)\subset G(a)$ is clearly a subgroup, in order to see $G'$ actually forms a crossed $\mathcal A$-subgroup of $G$, it is enough to check it is an $\mathcal A$-subset.
Note that since the subset \eqref{eq:prf:crsgrp-subgen:single} is closed under inverses, every element in $G'(a)$ can be written as products of elements of \eqref{eq:prf:crsgrp-subgen:single}.
For $\psi:b\to a\in\mathcal A$, $\varepsilon_i=\pm1$, and $\varphi_i:a\to a_0\in\mathcal A$, we have
\[
\psi^\ast(\varphi_1^\ast(x^{\varepsilon_1})\dots\varphi_k^\ast(x^{\varepsilon_k}))
=(\varphi_1\psi^{\varphi_2^\ast(x^{\varepsilon_2}_2)\dots\varphi_k^\ast(x^{\varepsilon_k}_k)})^\ast(x^{\varepsilon_1}_1)\dots(\varphi_k\psi)^\ast(x^{\varepsilon_k}_k)\ ,
\]
which shows $\psi^\ast:G(a)\to G(b)$ carries $G'(a)$ into $G'(b)$, and hence $G'\subset G$ is a crossed $\mathcal A$-subgroup.
On the other hand, the definition of $G'$ directly implies
\[
\left|\underline G'\right|
=\left|\coprod_{a\in\mathcal A}G'(a)\right|
\le \omega\times 2\times|\mathcal A|<\kappa
\]
so $G'\in\mathbf{CrsGrp}_{\mathcal A}^{<\kappa}$.

In general case, for each $x\in S$, the proof of the singleton case implies that there is a crossed $\mathcal A$-subgroup $G'_x\subset G$ so that $x\in\underline{G'_x}$ and $G'_x\in\mathbf{CrsGrp}_{\mathcal A}^{<\kappa}$.
We take $G'$ to be the image of the map
\[
G'_1\ast\dots\ast G'_n\to G
\]
of crossed $\mathcal A$-groups induced by the inclusions $G'_i\hookrightarrow G$.
Clearly $S\subset\underline G'$, and \cref{lem:crsgrp-kappa-cocompl} implies that $G'\in\mathbf{CrsGrp}_{\mathcal A}^{<\kappa}$.
This completes the proof.
\end{proof}

\begin{proof}[Proof of \cref{theo:crsgrp-locpres}]
In view of Corollary 2.47 in \cite{AdamekRosicky1994}, it suffices to show  the category $\mathbf{CrsGrp}_{\mathcal A}$ is accessible; i.e. there is a set $A$ of objects such that every object in $\mathbf{CrsGrp}_{\mathcal A}$ can be written as a $\kappa$-filtered colimit of objects from $A$.
This is a direct consequence of \cref{lem:crsgrp-kappa-s} and \cref{lem:crsgrp-subgen}.
\end{proof}

\section{Computation of the terminal object}
\label{sec:terminal}

In the previous section, we proved \cref{cor:crsgrp-compl} that asserts the category $\mathbf{CrsGrp}_{\mathcal A}$ has all small limits and colimits.
In fact, \cref{prop:crsgrp-colim}, \cref{cor:crsgrp-filtcolim}, and \cref{prop:crsgrp-limconn} provide relatively practical ways to compute colimits and limits except for the terminal object.
On the other hand, as for a few categories $\mathcal A$, the terminal crossed $\mathcal A$-group was constructed by hand.
For example, if $\mathcal A=\Delta$ or $\widetilde\Delta$, it is precisely the Hyperoctahedral crossed group $\mathfrak H$; see \cite{FL91}.
The goal of this section is to generalize their results and compute the terminal crossed $\mathcal A$-group for more general $\mathcal A$ including $\nabla$.
Note that, although the arguments are highly abstract, the reader should keep the concrete examples in the mind; such as $\mathcal A=\Delta$,$\widetilde\Delta$, or $\nabla$.

The key idea is \emph{orderings} on objects.

\begin{definition}
Let $\mathcal A$ be a category and $s\in\mathcal A$ an object.
Then, an \emph{internal co-relation} on $s$ is a tuple $(\bar s;\iota_0,\iota_1)$ of an object $\bar s\in\mathcal A$ and a jointly-epimorphic pair $\iota_0,\iota_1:s\to\bar s$ in $\mathcal A$.
By abuse of notation, we often denote it just by $\bar s$.
\end{definition}

If $\bar s$ is an internal co-relation on $s\in\mathcal A$, for each $a\in\mathcal A$, we have a map
\[
(\iota^\ast_0,\iota^\ast_1):\mathcal A(\bar s,a)\to\mathcal A(s,a)\times\mathcal A(s,a)
\]
which is injective since $(\iota_0,\iota_1)$ is jointly-epimorphic.
In other words, $\mathcal A(\bar s,a)$ is a (binary) \emph{relation} on $\mathcal A(s,a)$.
Explicitly, for two morphisms $\alpha_0,\alpha_1:s\rightrightarrows a$, they are connected by the relation if and only if there is a morphism $\bar\alpha:\bar s\to a$ such that $\alpha_i=\bar\alpha\iota_i$ for $i=0,1$.
Hence, every morphism $a\to b\in\mathcal A$ induces a map $\mathcal A(s,a)\to\mathcal A(s,b)$ which preserves the relation induced by $\bar s$.

\begin{definition}
Let $\mathcal A$ be a category.
An \emph{internal well-co-order} on an object $s\in\mathcal A$ is an internal co-relation $\bar s$ on $s$ such that the set $\mathcal A(s,a)$ is well-ordered by the induced relation for every $a\in\mathcal A$.
\end{definition}

\begin{example}
\label{ex:ord-well-coord}
Let $\mathbf{Ord}$ be the category of well-ordered sets and order-preserving maps.
In particular, $\mathbf{Ord}$ contains all the finite ordinals $\underline n=\{0,\dots,n-1\}$.
We have an obvious internal well-co-order on $\underline 1$ as
\[
\iota_0,\iota_1:\underline 1\rightrightarrows\underline 2\ .
\ ;\quad 0\mapsto 0,1
\]
Hence, every full subcategory of $\mathbf{Ord}$ containing the diagram isomorphic to the above one, such as $\Delta$ and $\widetilde\Delta$, $\underline 1$ is an internally well-co-ordered object.
Namely, the generators $[0]\in\Delta$ and $\langle1\rangle\in\widetilde\Delta$ are canonically well-co-ordered objects respectively in the way above.
\end{example}

\begin{example}
\label{ex:nabla-well-coord}
The object $\double\langle1\double\rangle\in\nabla$ admits a canonical internal co-order:
\[
\iota_0,\iota_1:\double\langle1\double\rangle\to\double\langle2\double\rangle
\ ;\quad 1\mapsto 1,2\ .
\]
It is easily verified that, for each $n\in\mathbb N$, the set $\nabla(\double\langle1\double\rangle,\double\langle n\double\rangle)$ together with the order induced by the internal co-order above is identified with the ordered set $\double\langle n\double\rangle$ itself.
Thus, $\double\langle1\double\rangle$ is a well-co-ordered object.
\end{example}

If $s\in\mathcal A$ admits an internal well-co-order $\bar s$, the corepresentable functor $\mathcal A(s,\blankdot)$ lifts to the functor
\[
\mathcal A(s,\blankdot):\mathcal A\to\mathbf{Ord}\ .
\]
Hence, we can make use of the following splendid property of the category $\mathbf{Ord}$ through this functor.

\begin{lemma}
\label{lem:totord-uniquemap}
Let $\varphi:A\to B$ be an order-preserving map between well-ordered sets such that the inverse image $\varphi^{-1}\{b\}$ is finite for each $b\in B$.
Suppose we are given a permutation $\sigma$ on $A$ with the composition $\varphi\sigma$ again order-preserving.
Then, we have $\varphi\sigma=\varphi$.
\end{lemma}
\begin{proof}
It suffices to prove the permutation $\sigma$ is restricted to each $A_b:=\varphi^{-1}\{b\}$.
Suppose we have $b\in B$ with $\sigma(A_b)\nsubset A_b$; in particular, we may assume $b$ is the minimum among such elements of $B$ since $B$ is well-ordered.
Take $a\in A_b$ such that $\varphi\sigma(a)\neq b$.
Note that we have $\varphi\sigma(a)>b$; otherwise, the minimality of $b$ implies $\sigma$ restricts to a permutation on $A_{\varphi\sigma(a)}$.
We have $a\notin A_{\varphi\sigma(a)}$ while $\sigma(\{a\}\cup A_{\varphi\sigma(a)})\subset A_{\varphi\sigma(a)}$, which contradicts to the injectivity of $\sigma$ since $A_{\varphi\sigma(a)}$ is finite.

Since $A_b$ and $\sigma^{-1}(A_b)$ are finite sets with the same cardinality, so are two subsets $A_b\setminus\sigma^{-1}(A_b)$ and $\sigma^{-1}(A_b)\setminus A_b$ of $A$.
The first one is non-empty, e.g. containing $a$, so we can take an element $a'\in\sigma^{-1}(A_b)\setminus A_b$.
The minimality of $b$ again implies $\varphi(a')>b=\varphi(a)$ so $a'>a$.
We however have
\[
\varphi\sigma(a')=b<\varphi\sigma(a)\ ,
\]
which contradicts to the assumption that $\varphi\sigma$ preserves the order.
\end{proof}

\begin{corollary}
\label{cor:permact-ord}
Let $A$ and $B$ be finite totally ordered (hence, well-ordered) set.
Then, the permutation group $\mathfrak S(B)$ on $B$ has a left action
\[
\mathfrak S(B)\times\mathbf{Pos}(A,B)\to\mathbf{Pos}(A,B)
\ ;\quad (\sigma,\varphi)\mapsto\varphi^\sigma
\]
on the set $\mathbf{Pos}(A,B)$ of order-preserving maps from $A$ to $B$ so that, for each $\sigma\in\mathfrak S(B)$ and $\varphi\in\mathbf{Pos}(A,B)$, $\varphi^\sigma:A\to B$ is the unique order-preserving map that the composition $\sigma\varphi$ factors through after a (not necessarily unique) permutation on $A$.
\end{corollary}
\begin{proof}
We may assume $B=\underline n=\{0,\dots,n-1\}$.
Since the composition $\sigma\varphi$ factors as
\[
A
\cong \varphi^{-1}\{\sigma^{-1}(0)\}\star\dots\star\varphi^{-1}\{\sigma^{-1}(n-1)\}
\to\underline n\ ,
\]
the existence of $\varphi^\sigma$ is obvious.
We show the uniqueness of $\varphi^\sigma$.
Say $\sigma\varphi=\varphi^\sigma\sigma'$, and suppose we have another factorization $\sigma\varphi=\psi\tau$.
Then, we have $\psi=\varphi^\sigma\sigma'\tau^{-1}$.
Since $A$ is finite, and since both $\varphi^\sigma$ and $\psi$ are order-preserving, \cref{lem:totord-uniquemap} implies $\psi=\varphi^\sigma$.
This guarantees the uniqueness of $\varphi^\sigma$.
Now, the associativity of the action easily follows from the property of $\varphi^\sigma$.
\end{proof}

In the rest of the section, we assume $\mathcal A$ to be a category such that
\begin{enumerate}[label={\rm(\roman*)}]
  \item $\mathcal A$ is locally finite; i.e. each hom-set $\mathcal A(a,b)$ is finite;
  \item it is equipped with a generator $s\in\mathcal A$, so the corepresentable functor $\mathcal A(s,\blankdot):\mathcal A\to\mathbf{Set}$ is by definition faithful;
  \item $s$ is internally well-co-ordered.
\end{enumerate}
As seen in \cref{ex:ord-well-coord} and \cref{ex:nabla-well-coord}, the examples of $\mathcal A$ include $\Delta$, $\widetilde\Delta$, and $\nabla$.

\begin{remark}
\label{rem:corel-refl-rep}
If $\mathcal A$ is a category satisfying the conditions above, the relation on each $\mathcal A(s,a)$ induced by the internal co-relation $\bar s$ is reflexive.
This means that every morphism $\alpha:s\to a$ determines a morphism $\bar\alpha:\bar s\to a$ such that $\bar\alpha\iota_0=\bar\alpha\iota_1=\alpha$.
In addition, since $\iota_0$ and $\iota_1$ are jointly epimorphic, such $\bar\alpha$ is unique.
Hence, we obtain a unique map
\begin{equation}
\label{eq:corel-refl}
\mathtt{refl}_a:\mathcal A(s,a)\to\mathcal A(\bar s,a)
\end{equation}
which is a common section of the precomposition maps with $\iota_0$ and $\iota_1$.
It is easily verified that $\mathsf{refl}_a$ is natural with respect to $a\in\mathcal A$, and Yoneda Lemma implies we have a map $\bar s\to s$ representing $\mathtt{refl}$.
Thus, for every morphism $\varphi:a\to b$, we have
\[
\varphi_\ast(\mathtt{refl}(\alpha)) = \mathtt{refl}(\varphi_\ast(\alpha))\ .
\]
\end{remark}

For $\mathcal A$ above, note that we can think of $\mathcal A(a,b)$ as a subset of the set
\[
\mathbf{Pos}(\mathcal A(s,a),\mathcal A(s,b))
\]
of order-preserving maps on which the group $\mathfrak S(\mathcal A(s,a))$ acts from the left.
We define a group $\mathfrak S_{\mathcal A}(a)$ by
\[
\mathfrak S_{\mathcal A}(a)
:= \left\{\sigma\in\mathfrak S(\mathcal A(s,a))\mid\forall b\in\mathcal A:\sigma(\mathcal A(b,a))\subset\mathcal A(b,a)\right\}\ .
\]
In particular, $\mathfrak S_{\mathcal A}(a)$ acts on $\mathcal A(b,a)$ from the left for each $b\in\mathcal A$.

\begin{example}
\label{eq:symg-wideDelta}
If $\mathcal A$ is a full subcategory of $\widetilde\Delta$ containing $\langle 1\rangle$ and $\langle 2\rangle$, and if we take $\langle 1\rangle$ as the generator with the canonical co-well-order, then we have
\[
\mathfrak S_{\mathcal A}(\langle n\rangle)
=\mathfrak S(\mathcal A(\langle 1\rangle,\langle n\rangle))
\cong\mathfrak S_n\ .
\]
In particular,
\[
\mathfrak S_{\widetilde\Delta}(\langle n\rangle)\cong\mathfrak S_n
\ ,\quad \mathfrak S_\Delta([n])\cong\mathfrak S_{n+1}\ .
\]
\end{example}

\begin{example}
\label{ex:symg-nabla}
In the case $\mathcal A=\nabla$, the evaluation map
\[
\nabla(\double\langle 1\double\rangle,\double\langle n\double\rangle)
\to\double\langle n\double\rangle
\ ;\quad\alpha\mapsto\alpha(1)
\]
is bijective, so we identify the two sets by the map.
We claim
\begin{equation}
\label{eq:ex:symg-nabla:assert}
\mathfrak S_\nabla(\double\langle n\double\rangle)
=\mathfrak S(\langle n\rangle)\times\mathfrak S(\{-\infty,\infty\})
\end{equation}
for each $n\in\mathbb N$ respecting the action on $\double\langle n\double\rangle$.
For the left-to-right inclusion, it suffices to see the subset $\{-\infty,\infty\}\subset\double\langle n\double\rangle$ is stable under the action of $\mathfrak S_\nabla(\double\langle n\double\rangle)$.
This follows from the observation that the set is the image of the unique map $\double\langle 0\double\rangle\to\double\langle n\double\rangle\in\nabla$.
To see the other direction, note that the action of $\mathfrak S(\langle n\rangle)$ on $\nabla(\double\langle m\double\rangle,\double\langle n\double\rangle)$ is given as follows: for $(\sigma,\varepsilon)\in\mathfrak S(\langle n\rangle)\times\mathfrak S(\{-\infty,\infty\})$, and for $\varphi\in\nabla(\double\langle m\double\rangle,\double\langle n\double\rangle)$, $\varphi^{(\sigma,\varepsilon)}$ is the map
\begin{equation}
\label{eq:ex:symg-nabla:act}
\begin{split}
\double\langle m\double\rangle
&\cong \varphi^{-1}\{\varepsilon(-\infty)\}\star\varphi^{-1}\{\sigma^{-1}(1)\}\star\dots\star\varphi^{-1}\{\sigma^{-1}(n)\}\star\varphi^{-1}\{\varepsilon(\infty)\} \\
&\to \{-\infty\}\star\{1\}\star\dots\star\{n\}\star\{\infty\}
\cong\double\langle n\double\rangle\ .
\end{split}
\end{equation}
Since the sets $\varphi^{-1}\{\varepsilon(-\infty)\}$ and $\varphi^{-1}\{\varepsilon(\infty)\}$ are, no matter what $\varepsilon$ is, non-empty, \eqref{eq:ex:symg-nabla:act} is a map in $\nabla$ for every pair $(\sigma,\varepsilon)$.
This shows the right hand side in \eqref{eq:ex:symg-nabla:assert} is contained in the left.
Hence, we obtain the result.
\end{example}

For each crossed $\mathcal A$-group $G$, the action of $G(a)$ on $\mathcal A(s,a)$ determines a group homomorphism
\begin{equation}
\label{eq:crsgrp-perm}
R_a:G(a)\to\mathfrak S_{\mathcal A}(a)
\end{equation}

\begin{lemma}
\label{lem:act-perm}
Let $\mathcal A$ be as above, and let $G$ be a crossed $\mathcal A$-group.
Then, the action of $G(a)$ on each $\mathcal A(b,a)$ factors through that of $\mathfrak S_{\mathcal A}(a)$ via the group homomorphism $R_a$ given in \eqref{eq:crsgrp-perm}.
\end{lemma}
\begin{proof}
For each $x\in G(a)$ and $\varphi:b\to a\in\mathcal A$, we have the following commutative square:
\[
\xymatrix{
  \mathcal A(s,b) \ar[r]^{\varphi_\ast} \ar[d]_{f^\ast(x)} & \mathcal A(s,a) \ar[d]^{x=R_a(x)} \\
  \mathcal A(s,b) \ar[r]^{\varphi^x_\ast} & \mathcal A(s,a) }
\]
By the definition of the action $\mathfrak S_{\mathcal A}(a)$ on $\mathcal A(b,a)$, which is essentially given in \cref{cor:permact-ord}, we have $\varphi^x=\varphi^{R_a(x)}$.
This is exactly the required result.
\end{proof}

\begin{lemma}
\label{lem:crsgrp-fibord}
Let $\mathcal A$ be as above, and let $G$ be a crossed $\mathcal A$-group.
Suppose $x\in G(a)$ and $\alpha\in\mathcal A(s,a)$.
Then, for every morphism $\varphi:b\to a\in\mathcal A$, the map
\[
\varphi^\ast(x):\mathcal A(s,b)\to\mathcal A(s,b)
\]
restricts to a map $(\varphi_\ast)^{-1}\{\alpha\}\to (\varphi^x_\ast)^{-1}\{\alpha^x\}$ which is either order-preserving or order-reversing, depending only on $x$ and $\alpha$ but not on $\varphi$.
\end{lemma}
\begin{proof}
Let $\varphi:b\to a\in\mathcal A$ be an arbitrary morphism, so we have a map $\varphi_\ast:\mathcal A(s,b)\to\mathcal A(s,a)$.
Take any two elements $\psi_0,\psi_1\in\mathcal A(s,b)$ with $\varphi_\ast(\psi_0)=\varphi_\ast(\psi_1)=\alpha$.
We may assume $\psi_0\le \psi_1$ so there is a morphism $\bar\psi:\bar s\to b$ with $\bar\psi\iota_0=\psi_0$ and $\bar\psi\iota_1=\psi_1$.
Since $\varphi\bar\psi\iota_0=\varphi\bar\psi\iota_1=\alpha$, we have $\varphi\bar\psi=\mathtt{refl}(\alpha)$, where $\mathtt{refl}$ is the map defined in \cref{rem:corel-refl-rep}, and the following diagram is commutative:
\[
\xymatrix{
  \mathcal A(s,\bar s) \ar[r]^{\bar\psi_\ast} \ar[d]_{\mathtt{refl}(\alpha)^\ast(x)} & \mathcal A(s,b) \ar[r]^{\varphi_\ast} \ar[d]_{\varphi^\ast(x)} & \mathcal A(s,a) \ar[d]^x \\
  \mathcal A(s,\bar s) \ar[r]^{\bar\psi^{\varphi^\ast(x)}} & \mathcal A(s,b) \ar[r]^{\varphi^x_\ast} & \mathcal A(s,a) }
\]
Hence, the order of two elements $\varphi^\ast(x)(\psi_0),\varphi^\ast(x)(\psi_1)\in\mathcal A(s,b)$ is determined by that of $\mathtt{refl}(\alpha)^\ast(x)(\iota_0)$ and $\mathtt{refl}(\alpha)^\ast(x)(\iota_1)$.
It clearly no longer depends on $\varphi$ nor elements $\psi_0,\psi_1\in (\varphi_\ast)^{-1}\{\alpha\}$, so we obtain the result.
\end{proof}

We now define a candidate for the terminal crossed $\mathcal A$-group.
For each morphism $\varphi:b\to a\in\mathcal A$, we define a map
\[
\beta_\varphi:C_2^{\times\mathcal A(s,a)}\to\mathfrak S(\mathcal A(s,b))
\]
as follows: for $\vec\varepsilon=(\varepsilon_i)_{i\in\mathcal A(s,a)}$, $\beta_\varphi(\vec\varepsilon)\in\mathfrak S(\mathcal A(s,b))$ is the unique permutation such that
\begin{itemize}
  \item it preserves the fibers of the map $\varphi_\ast:\mathcal A(s,b)\to\mathcal A(s,a)$;
  \item for each $i\in\mathcal A(s,a)$, the map $\beta_\varphi(\vec\varepsilon):(\varphi_\ast)^{-1}\{i\}\to(\varphi_\ast)^{-1}\{i\}$ is either order-preserving or order-reversing according to $\varepsilon_i\in C_2$.
\end{itemize}
The map $\beta_\varphi$ is obviously a group homomorphism.
Also, we define a map
\[
\varphi^\ast:\mathfrak S_{\mathcal A}(a)\to\mathfrak S(\mathcal A(s,b))
\]
so that $\varphi^\ast(\sigma)$ is the unique permutation satisfying
\begin{enumerate}[label={\rm(\roman*)}]
  \item the square below is commutative:
\[
\xymatrix{
  \mathcal A(s,b) \ar[r]^{\varphi_\ast} \ar[d]_{\varphi^\ast(\sigma)} & \mathcal A(s,a) \ar[d]^\sigma \\
  \mathcal A(s,b) \ar[r]^{\varphi^\sigma_\ast} & \mathcal A(s,a) }
\]
  \item $\varphi^\ast(\sigma):\mathcal A(s,b)\to\mathcal A(s,b)$ restricts to an order-preserving map
\[
\varphi^\ast(\sigma):(\varphi_\ast)^{-1}\{i\}\to (\varphi_\ast^x)^{-1}\{i^x\}
\]
for each $i\in\mathcal A(s,a)$.
\end{enumerate}
The action of $\mathfrak S_{\mathcal A}(a)$ on $\mathcal A(s,b)$ enables us to consider the semidirect product $\mathfrak S_{\mathcal A}(a)\ltimes C_2^{\times\mathcal A(s,a)}$, which is just a cartesian product $\mathfrak S_{\mathcal A}(a)\times C_2^{\times\mathcal A(s,a)}$ as a set together with the multiplication
\[
\left(\sigma;(\varepsilon_i)_{i\in\mathcal A(s,a)}\right)\left(\tau;(\zeta_i)_{i\in\mathcal A(s,a)}\right)
= \left(\sigma\tau;(\varepsilon_i\zeta_{\tau(i)})_{i\in\mathcal A(s,a)}\right)\ .
\]
We define
\begin{equation}
\label{eq:def-Weyl}
\mathfrak W_{\mathcal A}(a)
:= \left\{(\sigma;\vec\varepsilon)\in \mathfrak S_{\mathcal A}(a)\ltimes C_2^{\times\mathcal A(s,a)}\ \middle|\ \forall \varphi\in\mathcal A(b,a):\varphi^\ast(\sigma)\beta_\varphi(\vec\varepsilon)\in\mathfrak S_{\mathcal A}(b)\right\}\ .
\end{equation}

\begin{remark}
\label{rem:fast-char}
The permutation $\tau=\varphi^\ast(\sigma)\beta_\varphi(\vec\varepsilon)$ on $\mathcal A(s,b)$ appearing in \eqref{eq:def-Weyl} of the definition of $\mathfrak W_{\mathcal A}(a)$ is the permutation characterized by the following two properties:
\begin{enumerate}[label={\rm(\roman*)}]
  \item the square below is commutative:
\[
\xymatrix{
  \mathcal A(s,b) \ar[r]^{\varphi_\ast} \ar[d]_{\tau} & \mathcal A(s,a) \ar[d]^\sigma \\
  \mathcal A(s,b) \ar[r]^{\varphi^\sigma_\ast} & \mathcal A(s,a) }
\]
  \item the restriction $\tau:(\varphi_\ast)^{-1}\{i\}\to (\varphi^\sigma_\ast)^{-1}\{i^\sigma\}$ is either order-preserving or order-reversing according to $\varepsilon_i\in C_2$.
\end{enumerate}
\end{remark}

\begin{proposition}
\label{prop:Weyl-grp}
Let $\mathcal A$ be as above.
Then, the subset $\mathfrak W_{\mathcal A}(a)\subset \mathfrak S_{\mathcal A}(a)\ltimes C_2^{\times\mathcal A(s,a)}$ given in \eqref{eq:def-Weyl} has the following properties:
\begin{enumerate}[label={\rm(\arabic*)}]
  \item\label{req:Weyl-grp:subgrp} $\mathfrak W_{\mathcal A}(a)$ is closed under multiplication and the inverses; hence it is a subgroup;
  \item\label{req:Weyl-grp:presh} for each morphism $\varphi:b\to a\in\mathcal A$, define
\[
\varphi^\ast:\mathfrak W_{\mathcal A}(a)\to\mathfrak W_{\mathcal A}(b)
\ ;\quad (\sigma;\vec\varepsilon)\mapsto\left(\varphi^\ast(\sigma)\beta_\varphi(\vec\varepsilon);\varphi^\ast(\vec\varepsilon)\right)\ .
\]
Then, this defines a structure of $\mathcal A$-sets on $\mathfrak W_{\mathcal A}$.
\end{enumerate}
Moreover, the family $\{\mathfrak W_{\mathcal A}(a)\}_{a\in\mathcal A}$ forms a crossed $\mathcal A$-group.
\end{proposition}
\begin{proof}
We first show \ref{req:Weyl-grp:subgrp}.
Recall that, for $(\sigma;\vec\varepsilon),(\tau;\vec\zeta)\in\mathfrak W_{\mathcal A}(a)$, their multiplication is given by
\[
(\tau;\vec\zeta)(\sigma;\vec\varepsilon)
= (\tau\sigma;\sigma^\ast(\zeta)\cdot\vec\varepsilon)\ .
\]
On the other hand, for every morphism $\varphi:b\to a\in\mathcal A$, we have the following commutative diagram:
\[
\xymatrix{
  \mathcal A(s,b) \ar[r]^{\varphi_\ast} \ar[d]_{\varphi^\ast(\sigma)\beta_\varphi(\vec\varepsilon)} & \mathcal A(s,a) \ar[d]^\sigma \\
  \mathcal A(s,b) \ar[r]^{\varphi^\sigma_\ast} \ar[d]_{(\varphi^\sigma)^\ast(\tau)\beta_{\varphi^\sigma}(\vec\zeta)} & \mathcal A(s,a) \ar[d]^\tau \\
  \mathcal A(s,b) \ar[r]^{\varphi^{\tau\sigma}_\ast} & \mathcal A(s,a) }
\]
Verifying the conditions in \cref{rem:fast-char}, one will see the left vertical composition coincides with the permutation $\varphi^\ast(\tau\sigma)\beta_\varphi(\vec\zeta\cdot\sigma^\ast(\vec\varepsilon))$.
It follows that $(\vec\zeta;\tau)(\vec\varepsilon;\sigma)\in\mathfrak W_{\mathcal A}(a)$.
Closedness under the inverses is proved similarly.

As for the part \ref{req:Weyl-grp:presh}, note that the map $\varphi^\ast:\mathfrak W_{\mathcal A}(a)\to\mathfrak W_{\mathcal A}(b)$ is actually well-defined.
Indeed, for every morphism $\psi:c\to b\in\mathcal A$, we have a commutative diagram below.
\[
\xymatrix@C=5em{
  \mathcal A(s,c) \ar[r]^{\psi_\ast} \ar[d]_{\psi^\ast(\varphi^\ast(\sigma)\beta_\varphi(\vec\varepsilon))\beta_\psi(\varphi^\ast(\vec\varepsilon))} & \mathcal A(s,b) \ar[r]^{\varphi_\ast} \ar[d]^{\varphi^\ast(\sigma)\beta_\varphi(\vec\varepsilon)} & \mathcal A(s,a) \ar[d]^\sigma \\
  \mathcal A(s,c) \ar[r]^{\psi^{\varphi^\ast(\sigma)\beta_\varphi(\vec\varepsilon)}} & \mathcal A(s,b) \ar[r]^{\varphi^\sigma} & \mathcal A(s,a) }
\]
Then, similarly to the part \ref{req:Weyl-grp:subgrp}, one can check the conditions in \cref{rem:fast-char} to see the left vertical arrow equals to the permutation $(\varphi\psi)^\ast(\sigma)\beta_{\varphi\psi}(\vec\varepsilon)$, which belongs to $\mathfrak S_{\mathcal A}(c)$ since $(\sigma;\vec\varepsilon)\in\mathfrak W_{\mathcal A}(a)$.
This argument also shows the functoriality, so $\mathfrak W_{\mathcal A}$ is an $\mathcal A$-set.

It remains to show $\mathfrak W_{\mathcal A}$ is a crossed $\mathcal A$-group.
However, one can notice the proofs of \ref{req:Weyl-grp:subgrp} and \ref{req:Weyl-grp:presh} above also show $\mathfrak W_{\mathcal A}$ satisfies the two conditions of crossed groups respectively.
\end{proof}

Let $\mathcal A$ be as above, and let $G$ be a crossed $\mathcal A$-group.
We have a group homomorphism $R_a:G(a)\to\mathfrak S_{\mathcal A}(a)$ given in \eqref{eq:crsgrp-perm}.
Unfortunately, it is not a map of $\mathcal A$-sets in general.
Nevertheless, it turns out that the only obstruction is \emph{parities}, so we can put that information on the codomain of $R_a$.
More precisely, define a map $\vec\varepsilon_a:G(a)\to C_2^{\times\mathcal A(s,a)}$ as follows: recall that for each $x\in G(a)$ and each $\alpha\in\mathcal A(s,a)$, the square below is commutative:
\[
\xymatrix@C=4em{
  \mathcal A(s,\bar s) \ar[r]^{\mathtt{refl}(\alpha)_\ast} \ar[d]_{\mathtt{refl}(\alpha)^\ast(x)} & \mathcal A(s,a) \ar[d]^x \\
  \mathcal A(s,\bar s) \ar[r]^{\mathtt{refl}(\alpha^x)_\ast} & \mathcal A(s,a) }
\]
We set $\vec\varepsilon_a(x)= \left(\varepsilon_a(x)_\alpha\right)_{\alpha\in\mathcal A(s,a)}$ by
\[
\varepsilon_a(x)_\alpha :=
\begin{cases}
1 &\quad \mathtt{refl}(\alpha)^\ast(x)(\iota_0)< \mathtt{refl}(\alpha)^\ast(x)(\iota_1) \\
-1 &\quad \mathtt{refl}(\alpha)^\ast(x)(\iota_0)> \mathtt{refl}(\alpha)^\ast(x)(\iota_1)\ .
\end{cases}
\]

\begin{lemma}
\label{lem:tilde-rho}
Let $\mathcal A$ and $G$ be as above.
Then, for each $a\in\mathcal A$ and $x\in G(a)$, the pair
\[
(R_a(x);\vec\varepsilon_a(x))
\in \mathfrak S(\mathcal A(s,a))\ltimes C_2^{\times\mathcal A(s,a)}
\]
belongs to the subgroup $\mathfrak W_{\mathcal A}(a)$ defined in \eqref{eq:def-Weyl}.
Moreover, the induced maps
\[
\tilde R_a:=(R_a;\vec\varepsilon_a):G(a)\to\mathfrak W_{\mathcal A}(a)
\]
forms a map $G\to\mathfrak W_{\mathcal A}$ of crossed $\mathcal A$-groups.
\end{lemma}
\begin{proof}
Verifying the conditions in \cref{rem:fast-char}, one can observe that, for each $x\in G(a)$, and for each $\varphi:b\to a\in\mathcal A$, we have
\begin{equation}
\label{eq:prf:tilde-rho:perm}
R_a(\varphi^\ast(x))
= \varphi^\ast(R_a(x))\cdot \beta_\varphi(\vec\varepsilon_a(x))\ .
\end{equation}
The left hand side clearly belongs to $\mathfrak S_{\mathcal A}(a)$, this implies $(R_a(x);\vec\varepsilon_a(x))\in\mathfrak W_{\mathcal A}(a)$ by the definition \eqref{eq:def-Weyl} of $\mathfrak W_{\mathcal A}$.
On the other hand, it is also verified that
\begin{equation}
\label{eq:prf:tilde-rho:sign}
\vec\varepsilon_a(\varphi^\ast(x))
= \varphi^\ast(\vec\varepsilon_a(x))\ .
\end{equation}
The equation \eqref{eq:prf:tilde-rho:sign} and \eqref{eq:prf:tilde-rho:perm} imply that $\tilde R_a=(R_a;\vec\varepsilon_a)$ is natural with respect to $a\in\mathcal A$, so it defines a map $\tilde R:G\to\mathfrak W_{\mathcal A}$ of $\mathcal A$-sets.

It is obvious that $\tilde R:G\to\mathfrak W_{\mathcal A}$ respects the action on each $\mathcal A(a,b)$, so it remains to prove it is a degreewise group homomorphism.
Since the preservation of the unit is straightforward, it suffices to show, for each $a\in\mathcal A$, $\tilde R_a:G(a)\to\mathfrak W_{\mathcal A}(a)$ preserves multiplications.
For each $x,y\in G(a)$, and for each $\alpha\in\mathcal A(s,a)$, we have a commutative diagram below:
\begin{equation}
\label{diag:prf:tilde-rho:signmul}
\xymatrix@C=4em{
  \mathcal A(s,\bar s) \ar[r]^{\mathtt{refl}(\alpha)_\ast} \ar[d]_{\mathtt{refl}(\alpha)^\ast(x)} & \mathcal A(s,a) \ar[d]^x \\
  \mathcal A(s,\bar s) \ar[r]^{\mathtt{refl}(\alpha^x)_\ast} \ar[d]_{\mathtt{refl}(\alpha^x)^\ast(y)} & \mathcal A(s,a) \ar[d]^y \\
  \mathcal A(s,\bar s) \ar[r]^{\mathtt{refl}(\alpha^{yx})_\ast} & \mathcal A(s,a) }
\end{equation}
One can deduce from \eqref{diag:prf:tilde-rho:signmul} that
\[
\varepsilon_a(yx)_\alpha = \varepsilon_a(y)_{\alpha^x}\varepsilon_a(x)\ ,
\]
which implies $\vec\varepsilon_a(yx)=R_a(x)_\ast(\vec\varepsilon_a(y))\vec\varepsilon_a(x)$.
Thus, we obtain
\[
\begin{split}
\tilde R(yx)
&= (R_a(yx);\vec\varepsilon_a(yx)) \\
&= (R_a(y)R_a(x);R_a(x)_\ast(\vec\varepsilon_a(y))\vec\varepsilon_a(x)) \\
&= (R_a(y);\vec\varepsilon_a(y))\cdot(R_a(x);\vec\varepsilon_a(x))
\end{split}
\]
so that $R_a:G(a)\to\mathfrak W_{\mathcal A}(a)$ is a group homomorphism.
\end{proof}

\begin{theorem}[cf. Theorem 1.4 in \cite{Kra87}]
\label{theo:crsgrp-Weyl}
Let $\mathcal A$ be a small locally finite category equipped with an internally well-co-ordered generator $s\in\mathcal A$.
Then, the Weyl crossed $\mathcal A$-group $\mathfrak W_{\mathcal A}$ is a terminal object in the category $\mathbf{CrsGrp}_{\mathcal A}$.
\end{theorem}
\begin{proof}
By virtue of \cref{lem:tilde-rho}, it suffices to show that, for each crossed $\mathcal A$-group, the map $\tilde R:G\to\mathfrak W_{\mathcal A}$ is the unique map of crossed $\mathcal A$-group for each $G\in\mathbf{CrsGrp}_{\mathcal A}$.
Since $\mathfrak W_{\mathcal A}(a)$ is, as a set, a subset of the direct product
\[
\mathfrak S_{\mathcal A}(a)\times C_2^{\mathcal A(s,a)}
\]
for each $a\in\mathcal A$, an $\mathcal A$-map $f:G\to\mathfrak W_{\mathcal A}$ is determined by maps
\[
f_{\mathsf{perm}}:G(a)\to\mathfrak S_{\mathcal A}(a)
\ ,\quad
f_{\mathsf{sign}}=(f^{(\alpha)}_{\mathsf{sign}})_{\alpha\in\mathcal A(s,a)}:G(a)\to C_2^{\times\mathcal A(s,a)}
\]
for each $a\in\mathcal A$.
If $f$ is a map of crossed $\mathcal A$-group, the map $f_{\mathsf{perm}}$ clearly has to be the one associated to the action of $G(a)$ on $\mathcal A(s,a)$.
On the other hand, each $f^{(\alpha)}_{\mathsf{sign}}$ is also determined automatically thanks to the naturality of $f$ and \cref{lem:crsgrp-fibord}.
It follows that $\tilde R$ is the only map into $\mathfrak W_{\mathcal A}$.
\end{proof}

\begin{corollary}
\label{cor:crsgrp-W/A}
Let $\mathcal A$ be as above.
Suppose $G$ is an $\mathcal A$-set which is equipped with a degreewise group structure.
Then, the following data are equivalent:
\begin{enumerate}[label={\rm(\alph*)}]
  \item left actions of $G(a)$ on $\mathcal A(b,a)$ for $a,b\in\mathcal A$ which exhibit $G$ as a crossed $\mathcal A$-group;
  \item a map $G\to\mathfrak W_{\mathcal A}$ of $\mathcal A$-sets which is a degreewise group homomorphism.
\end{enumerate}
\end{corollary}

\begin{corollary}
\label{cor:crsgrpmap-W/A}
Let $\mathcal A$ be as above, and let $G$ and $H$ be crossed $\mathcal A$-groups.
Then, a map $f:G\to H$ of $\mathcal A$-sets which is a degreewise group homomorphism is a map of crossed $\mathcal A$-groups if and only if the triangle below is commutative:
\[
\xymatrix{
  G \ar[rr]^f \ar[dr]_{\tilde R} && H \ar[dl]^{\tilde R} \\
  & \mathfrak W_{\mathcal A} & }
\]
\end{corollary}

\begin{example}
In the case $\mathcal A=\widetilde\Delta$, since it is a full subcategory of $\mathbf{Ord}$, we have
\[
\begin{gathered}
\mathfrak S_{\widetilde\Delta}(\langle n\rangle)
=\mathbf{Pos}(\widetilde\Delta(\langle 1\rangle,\langle n\rangle))
\cong\mathfrak S_n \\
\mathfrak W_{\widetilde\Delta}(\langle n\rangle)
=\mathfrak S_n\ltimes C_2^{\times n}
\cong H_n\ .
\end{gathered}
\]
It is easily verified that these induce an isomorphism from $\mathfrak W_{\widetilde\Delta}$ to the hyperoctahedral crossed simplicial group $\mathfrak H$ described in \cref{ex:crsgrp-simpWeyl}
\end{example}

\begin{example}
\label{ex:crsgrp-intWeyl}
Recall that, in the case $\mathcal A=\nabla$, we saw in \cref{ex:symg-nabla} that $\mathfrak S_\nabla(\double\langle n\double\rangle)\cong\mathfrak S(\langle n\rangle)\times\mathfrak S(\{-\infty,\infty\})$.
We claim an element of $\mathfrak S_\nabla(\double\langle n\double\rangle)\ltimes C_2^{\times\double\langle n\double\rangle}$ of the form
\[
x=((\sigma,\theta);\varepsilon_{-\infty},\varepsilon_1\dots,\varepsilon_n,\varepsilon_\infty)
\]
for $(\sigma,\theta)\in\mathfrak S(\langle n\rangle)\times\mathfrak S(\{-\infty,\infty\})$ and $\varepsilon_i\in C_2$ belongs to $\mathfrak W_\nabla(\double\langle n\double\rangle)$ if and only if $\theta=\varepsilon_{-\infty}=\varepsilon_\infty$ under the canonical identification $\mathfrak S(\{-\infty,\infty\})\cong C_2$.
For a map $\varphi:\double\langle m\double\rangle\to\double\langle n\double\rangle\in\nabla$, the permutation $\varphi^\ast(x)$ restricts to bijections
\[
\varphi^{-1}\{-\infty\}\to\varphi^{-1}\{\theta(-\infty)\}
\ ,\quad \varphi^{-1}\{\infty\}\to\varphi^{-1}\{\theta(\infty)\}
\]
which are either order-preserving or order-reversing according to $\varepsilon_{-\infty}$ and $\varepsilon_\infty$ respectively.
On the other hand, in view of the computation of \cref{ex:symg-nabla}, $\varphi^\ast(x)$ belongs to $\mathfrak S_\nabla(\double\langle m\double\rangle)$ if and only if it preserves the subset $\{-\infty,\infty\}$.
It is easily seen that this happens for every $\varphi:\double\langle m\double\rangle\to\double\langle n\double\rangle$ precisely when the elements $\theta$ and $\varepsilon_{\pm\infty}$ all coincide.
As a consequence, we obtain an isomorphism
\[
\mathfrak W_\nabla(\double\langle n\double\rangle)
\cong H_n\times C_2\ .
\]
\end{example}

We finally note that the explicit computation of terminal crossed $\mathcal A$-group leads to a classification.

\begin{proposition}[cf. Theorem 3.6 in \cite{FL91}]
\label{prop:crsgrp-clsfy}
Let $\mathcal A$ be an arbitrary small category.
Then, for every crossed $\mathcal A$-group $G$, there is a sequence
\begin{equation}
\label{eq:crsgrp-clsfy:ex}
G^{\mathsf{nc}} \hookrightarrow G \twoheadrightarrow G^{\mathsf{red}}
\end{equation}
of maps of crossed $\mathcal A$-groups such that
\begin{enumerate}[label={\rm(\roman*)}]
  \item the sequence \eqref{eq:crsgrp-clsfy:ex} is degreewisely a short exact sequence of groups;
  \item $G^{\mathsf{nc}}$ is a (non-crossed) $\mathcal A$-group;
  \item $G^{\mathsf{red}}$ is a crossed $\mathcal A$-subgroup of the terminal crossed $\mathcal A$-group.
\end{enumerate}
Moreover, the sequence \eqref{eq:crsgrp-clsfy:ex} extends to a functor from $\mathbf{CrsGrp}_{\mathcal A}$ into the category of degreewise short exact sequences in $\mathbf{CrsGrp}_{\mathcal A}$.
\end{proposition}
\begin{proof}
Put $\mathfrak T_{\mathcal A}$ the terminal crossed $\mathcal A$-group, and set $G^{\mathsf{red}}$ to be the image of the unique map $G\to\mathfrak T_{\mathcal A}$.
Then, we can define $G^{\mathsf{nc}}$ by the following pullback square:
\[
\xymatrix{
  G^{\mathsf{nc}} \ar[r] \ar[d] \ar@{}[dr]|(.4)\pbcorner & \ast \ar[d] \\
  G \ar[r] & G^{\mathsf{red}} }
\]
The required properties are verified easily.
\end{proof}

\begin{example}
\label{ex:simpcrs-clsfy}
In Proposition 3.5 in \cite{FL91}, there is a complete list of crossed simplicial subgroups of $\mathfrak W_\Delta\cong\mathfrak H$ as in \cref{tb:simpcrs-clsfy}.
\begin{table}[htbp]
\centering
\begin{tabular}{c|c|c}
  \textbf{name} & \textbf{symbol} & \textbf{$n$-th group} \\\hline
  Trivial & $*$ & $1$ \\
  Reflexive & $C_2$ & $C_2$ \\
  Cyclic & $\Lambda$ & $C_{n+1}$ \\
  Dihedral & $\mathfrak D$ & $D_{n+1}$ \\
  Symmetric & $\mathfrak S$ & $\mathfrak S_{n+1}$ \\
  Reflexosymmetric & $\widetilde{\mathfrak S}$ & $\mathfrak S_{n+1}\times C_2$ \\
  Weyl (Hyperoctahedral) & $\mathfrak W_\Delta\cong\mathfrak H$ & $H_{n+1}$
\end{tabular}
\caption{The crossed simplicial subgroups of $\mathfrak W_\Delta$}
\label{tb:simpcrs-clsfy}
\end{table}
\end{example}

\begin{example}
We will see in \cref{ex:Ran-aug} in \cref{sec:basechange} that the embedding $\Delta\hookrightarrow\widetilde\Delta$ induces a fully faithful embedding
\[
\mathbf{CrsGrp}_\Delta\hookrightarrow\mathbf{CrsGrp}_{\widetilde\Delta}
\]
which sends $\mathfrak W_\Delta$ to $\mathfrak W_{\widetilde\Delta}$.
Hence, we obtain the same list of augmented crossed simplicial subgroups of $\mathfrak W_{\widetilde\Delta}$ as \cref{tb:simpcrs-clsfy} while the indices are shifted by $1$ because of the identification $[n]\cong\langle n+1\rangle$.
\end{example}

In the appendix, we will compute all the crossed interval subgroups of $\mathfrak W_\nabla$.

\section{Crossed groups as monoid objects}
\label{sec:monobj}

We proved \cref{theo:crsgrp-locpres} that asserts the category $\mathbf{CrsGrp}_{\mathcal A}$ is locally presentable for each small category $\mathcal A$ by highly abstract argument, and it was for the sake of the existence of terminal crossed groups.
On the other hand, we also proved the latter independently in a more explicit way in \cref{theo:crsgrp-Weyl} for some special categories $\mathcal A$.
Combining this with \cref{prop:crsgrp-colim} and \cref{prop:crsgrp-limconn}, which are proved in more or less constructive ways, one can recover the completeness and the cocompleteness of $\mathbf{CrsGrp}_{\mathcal A}$ mentioned in \cref{cor:crsgrp-compl}.
In this section, we are going further: we see the category $\mathbf{CrsGrp}_{\mathcal A}$ is even \emph{algebraic}, in some sense, over a presheaf category, which guarantees the locally presentability.
More precisely, for a crossed $\mathcal A$-group $G$, we have the forgetful functor
\begin{equation}
\label{eq:crsgrp-forget/}
\mathbf{CrsGrp}^{/G}_{\mathcal A}\to\mathbf{Set}^{/G}_{\mathcal A}
\end{equation}
between slice categories.
One has the category $\mathbf{CrsGrp}_{\mathcal A}$ on the left hand side when he takes $G$ to be the terminal crossed $\mathcal A$-group, say $G=\mathfrak T_{\mathcal A}$.
Note that though the functor \eqref{eq:crsgrp-forget/} really forgets degreewise group structures, it remembers the actions on each hom-set $\mathcal A(a,b)$ through that of $G$, while the category $\mathbf{Set}^{/G}_{\mathcal A}$ is just a presheaf topos.
This suggests that the \emph{true} underlying category of $\mathbf{CrsGrp}_{\mathcal A}$ should be not $\mathbf{Set}_{\mathcal A}$ but the slice category $\mathbf{Set}^{/\mathfrak T_{\mathcal A}}_{\mathcal A}$.
Throughout this section, we fix a small category $\mathcal A$ and a crossed $\mathcal A$-group $G$ and aim to see \eqref{eq:crsgrp-forget/} has a \emph{good} left adjoint.

\begin{notation}
For each object $X\in\mathbf{Set}^{/G}_{\mathcal A}$, say $p:X\to G$ is the structure map, we consider its \emph{action} on hom-sets as
\[
X(b)\times\mathcal A(a,b)\to\mathcal A(a,b)
\ ;\quad (x,\varphi)\mapsto \varphi^x := \varphi^{p(x)}
\]
for $a,b\in\mathcal A$ even though $X(b)$ is no longer a group.
\end{notation}

To begin with, we introduce the following construction.

\begin{definition}
For $K\in\mathbf{Set}_{\mathcal A}$ and $X\in\mathbf{Set}^{/G}_{\mathcal A}$, we define an $\mathcal A$-set $K\rtimes_G X$ as follows:
\begin{itemize}
  \item for each $a\in\mathcal A$, we set $(K\rtimes_G X)(a):=K(a)\times X(a)$;
  \item for each $\varphi:a\to b\in\mathcal A$, we set
\[
\varphi^\ast:(K\rtimes_G X)(b)\to (K\rtimes_G X)(a)
\ ;\quad (k,x)\mapsto ((\varphi^x)^\ast(k),\varphi^\ast(x))\ .
\]
\end{itemize}
\end{definition}

\begin{remark}
The operation $\rtimes$ was originally introduced by Krasauskas in Definition 2.1 in \cite{Kra87} in the case $\mathcal A=\Delta$ and $G=\mathfrak W_\Delta$.
\end{remark}

To see $K\rtimes_G X$ above actually defines an $\mathcal A$-set, it is convenient to consider the map
\begin{equation}
\label{eq:semicrs}
\operatorname{crs}_X:X(b)\times\mathcal A(a,b)\to\mathcal A(a,b)\times X(a)
\ ;\quad (x,\varphi)\mapsto (\varphi^x,\varphi^\ast(x))
\end{equation}
for each $X\in\mathbf{Set}^{/G}_{\mathcal A}$.
Similarly to the case of crossed groups, we have the following commutative diagram:
\begin{equation}
\label{eq:semicrs-commsq}
\vcenter{
  \xymatrix{
    X(c)\times\mathcal A(b,c)\times\mathcal A(a,b) \ar[r]_{\vec{\operatorname{crs}_X}} \ar[d]^{\mathrm{id}\times\mathrm{comp}} & \mathcal A(b,c)\times\mathcal A(a,b)\times X(a) \ar[d]^{\mathrm{comp}\times\mathrm{id}} \\
    X(c)\times\mathcal A(a,c) \ar[r]^{\operatorname{crs}_X} & \mathcal A(a,c)\times X(a) }}
\end{equation}
Note that the $\mathcal A$-set structure on $K\rtimes_G X$ is given by
\begin{equation}
\label{eq:semicrs-str}
\begin{split}
(K\rtimes_G X)(b)\times\mathcal A(a,b)
&= K(b)\times X(b)\times\mathcal A(a,b) \\
&\xrightarrow{\mathrm{id}\times\operatorname{crs}_X} K(b)\times\mathcal A(a,b)\times X(a) \\
&\xrightarrow{\mathrm{act_K}\times\mathrm{id}} K(a)\times X(a) \\
&= (K\rtimes_G X)(a)\ .
\end{split}
\end{equation}
Then, combining \eqref{eq:semicrs-commsq} and \eqref{eq:semicrs-str}, one can verify $K\rtimes_G X$ is actually an $\mathcal A$-set.

\begin{lemma}
\label{lem:semicrs-nat}
If two objects $a,b\in\mathcal A$ are fixed, the map $\operatorname{crs}$ given in \eqref{eq:semicrs} is natural with respect to $X\in\mathbf{Set}^{/G}_{\mathcal A}$.
\end{lemma}
\begin{proof}
Suppose $f:X\to Y\in\mathbf{Set}^{/G}_{\mathcal A}$.
We have to show the square below commutes:
\[
\xymatrix{
  X(b)\times\mathcal A(a,b) \ar[r]^{\operatorname{crs}_X} \ar[d]_{f\times\mathrm{id}} & \mathcal A(a,b)\times X(a) \ar[d]^{\mathrm{id}\times f} \\
  Y(b)\times\mathcal A(a,b) \ar[r]^{\operatorname{crs}_Y} & \mathcal A(a,b)\times Y(a) }
\]
For $(x,\varphi)\in X(b)\times\mathcal A(a,b)$, we have
\[
\begin{gathered}
(\mathrm{id}\times f)\circ\operatorname{crs}_X(x,\varphi) = (\varphi^x,f\varphi^\ast(x))\ , \\
\operatorname{crs}_Y\circ(f\times\mathrm{id})(x,\varphi) = (\varphi^{f(x)},\varphi^\ast f(x))\ .
\end{gathered}
\]
Since $f$ is a map of $\mathcal A$-set over $G$, the actions of $x$ and $f(x)$ on $\mathcal A(a,b)$ agree with each other, so the required result follows.
\end{proof}

\begin{corollary}
\label{cor:semicrs-funct}
The assignment $(K,X)\to K\rtimes_G X$ defines a functor
\[
(\rtimes_G):\mathbf{Set}_{\mathcal A}\times\mathbf{Set}^{/G}_{\mathcal A}\to\mathbf{Set}_{\mathcal A}\ .
\]
\end{corollary}

The key observation is that we can lift the functor $(\rtimes_G)$ to a monoidal structure on $\mathbf{Set}^{/G}_{\mathcal A}$.
For this, note that the degreewise multiplication gives rise to an $\mathcal A$-map
\[
\mu:G\rtimes_G G\to G\ .
\]
Indeed, for $\varphi:a\to b\in\mathcal A$, the condition \ref{condCG:cdist} in the definition of crossed $\mathcal A$-groups implies the following square is commutative:
\[
\xymatrix{
  (G\rtimes_G G)(b) \ar[d]_{\varphi^\ast} \ar[r]^-\mu & G(b) \ar[d]^{\varphi^\ast} \\
  (G\rtimes_G G)(a) \ar[r]^-\mu & G(a) }
\]
Hence, we can define a functor
\[
(\rtimes_G):\mathbf{Set}^{/G}_{\mathcal A}\times\mathbf{Set}^{/G}_{\mathcal A}\to\mathbf{Set}^{/G}_{\mathcal A}
\]
by
\[
(X,Y)\mapsto (X\rtimes_G Y\to G\rtimes_G G\xrightarrow\mu G)\ .
\]

\begin{proposition}
\label{prop:semicrs-monoidal}
The functor
\[
(\rtimes_G):\mathbf{Set}^{/G}_{\mathcal A}\times\mathbf{Set}^{/G}_{\mathcal A}\to\mathbf{Set}^{/G}_{\mathcal A}
\]
given above is a monoidal structure on $\mathbf{Set}^{/G}_{\mathcal A}$ where the unit is the terminal $\mathcal A$-set $\ast$ with the unit map $\ast\to G$.
Moreover, the monoidal structure is biclosed; i.e. there are functors
\[
\mathrm{Hom}^{\mathsf L}_G,\mathrm{Hom}^{\mathsf R}_G:(\mathbf{Set}^{/G}_{\mathcal A})^\opposite\times\mathbf{Set}^{/G}_{\mathcal A}\to\mathbf{Set}^{/G}_{\mathcal A}
\]
together with natural isomorphisms
\[
\begin{split}
\mathbf{Set}^{/G}_{\mathcal A}(X\rtimes_G Y,Z)
&\cong \mathbf{Set}^{/G}_{\mathcal A}(X,\mathrm{Hom}^{\mathsf R}_G(Y,Z)) \\
&\cong \mathbf{Set}^{/G}_{\mathcal A}(Y,\mathrm{Hom}^{\mathsf L}_G(X,Z))\ .
\end{split}
\]
\end{proposition}
\begin{proof}
We first show that, for $X,Y,Z\in\mathbf{Set}^{/G}_{\mathcal A}$, the degreewise canonical identification actually gives an isomorphism
\[
(X\rtimes_G Y)\rtimes_G Z \cong Z\rtimes_G (Y\rtimes_G Z)\ .
\]
Since it is clearly a degreewise bijection, it suffices to show it is actually a map of $\mathcal A$-sets over $G$.
Suppose $\varphi:a\to b\in\mathcal A$ is a morphism in $\mathcal A$.
Then, an easy computation shows that both maps
\[
\begin{gathered}
\varphi^\ast:((X\rtimes_G Y)\rtimes_G Z)(b)\to ((X\rtimes_G Y)\rtimes_G Z)(a) \\
\varphi^\ast:(X\rtimes_G (Y\rtimes_G Z))(b)\to (X\rtimes_G (Y\rtimes_G Z))(a)
\end{gathered}
\]
are identified with the map $\varphi^\ast:X(b)\times Y(b)\times Z(b)\to X(a)\times Y(a)\times Z(a)$ given by
\[
\varphi^\ast(x,y,z)
=(((\varphi^y)^z)^\ast(x),(\varphi^z)^\ast(y),z)\ .
\]
Thus, we obtain a canonical identification $(X\rtimes_G Y)\rtimes_G Z\cong X\rtimes_G(Y\rtimes_G Z)$ as $\mathcal A$-sets.
Actually it is an isomorphism over $G$; indeed, the structure maps into $G$ are given by the common formula
\[
X(a)\times Y(a)\times Z(a)\to G(a)
\;\quad (x,y,z)\mapsto p(x)q(y)r(z)\ ,
\]
where $p:X\to G$, $q:Y\to G$, and $r:Z\to G$ are the structure maps.
Therefore, we obtain an associativity isomorphism for the functor $\rtimes_G$.
The unitality of $\ast$ is easily verify, so the first assertion follows.

To see the monoidal structure is biclosed, note that the category $\mathbf{Set}^{/G}_{\mathcal A}$ is a presheaf topos and so locally presentable.
Hence, by General Adjoint Functor Theorem, it suffices to show the functor $\rtimes_G$ preserves arbitrary small colimits in each variable.
Notice that colimits in $\mathbf{Set}^{/G}_{\mathcal A}$ are computed in the category $\mathbf{Set}_{\mathcal A}$ and so agree with the degreewise ones.
Now, the functor $(\rtimes_G)$ is degreewisely just the cartesian product, so the problem is reduced to the case $\mathcal A=\ast$ where the result is obvious.
\end{proof}

\begin{remark}
We have an explicit description of $\mathcal A$-sets $\mathrm{Hom}^{\mathsf R}_G(Y,Z)$ and $\mathrm{Hom}^{\mathsf L}_G(X,Z)$ as follows:
\[
\begin{gathered}
\mathrm{Hom}^{\mathsf R}_G(Y,Z)(a)
= \coprod_{\mathcal A[a]\to G} \mathbf{Set}_{\mathcal A}(\mathcal A[a]\rtimes_G Y,Z) \\
\mathrm{Hom}^{\mathsf L}_G(X,Z)(a)
= \coprod_{\mathcal A[a]\to G} \mathbf{Set}_{\mathcal A}(X\rtimes_G\mathcal A[a],Z)
\end{gathered}
\]
\end{remark}

\begin{remark}
Similarly to \cref{prop:semicrs-monoidal}, one can also prove that the functor $\rtimes_G:\mathbf{Set}_{\mathcal A}\times\mathbf{Set}^{/G}_{\mathcal A}\to\mathbf{Set}_{\mathcal A}$ given in \cref{cor:semicrs-funct} defines a right action of the monoidal category $(\mathbf{Set}^{/G}_{\mathcal A},\rtimes_G)$ on the category $\mathbf{Set}_{\mathcal A}$.
Note that, in the case $\mathcal A=\Delta$, this functor was discussed in Section 4 and 5 in \cite{FL91}, where they wrote $F_G(X):=X\rtimes_G G$ for $X\in\mathbf{Set}_\Delta$ and a crossed simplicial group $G$.
\end{remark}

\begin{lemma}
\label{lem:coeff-monoidal}
Let $G\to H$ be a map of crossed $\mathcal A$-groups.
Then, the induced functor $\mathbf{Set}^{/G}_{\mathcal A}\to\mathbf{Set}^{/H}_{\mathcal A}$ is monoidal with respect to monoidal structures $\rtimes_G$ and $\rtimes_H$.
\end{lemma}
\begin{proof}
For each $X,Y\in\mathbf{Set}^{/G}_{\mathcal A}$, $X\rtimes_G Y$ and $X\rtimes_H Y$ are clearly identical as $\mathcal A$-sets.
In addition, since $G\to H$ is a map of crossed $\mathcal A$-groups, the square
\[
\xymatrix{
  X\rtimes_G Y \ar@{=}[r] \ar[d] & X\rtimes_H Y \ar[d] \\
  G \ar[r] & H }
\]
is commutative.
Hence, we obtain the required result.
\end{proof}

We are interested in \emph{monoid objects} in the category $\mathbf{Set}^{/G}_{\mathcal A}$ with respect to the monoidal structure $\rtimes_G$.
Recall that a monoid object in $\mathbf{Set}^{/G}_{\mathcal A}$ is an object $M$ equipped with two morphisms
\[
\begin{gathered}
\eta:\ast\to M \\
\mu:M\rtimes_S M\to M
\end{gathered}
\]
satisfying the ordinary conditions of monoids, namely the \emph{associativity} and the \emph{unitality}.
The next lemma shows crossed $\mathcal A$-groups are examples of monoid objects.

\begin{lemma}
\label{lem:crsgrp-monobj}
Let $H$ be a crossed $\mathcal A$-group over $G$; i.e. a crossed $\mathcal A$-group equipped with a map $H\to G$ of crossed $\mathcal A$-groups.
Then, the maps $\eta_H:\ast\to H$ and $\mu_H:H\rtimes_G H\to H$ given by
\[
\begin{gathered}
\eta_H:\ast\to H(a)
\ ;\quad \ast\mapsto e_a \\
\mu_H:(H\rtimes_G H)(a)\to H(a)
\ ;\quad (x,y) \mapsto xy
\end{gathered}
\]
are maps of $\mathcal A$-sets over $G$.
Moreover, they exhibit $G$ as a monoid object in $\mathbf{Set}^{/G}_{\mathcal A}$ with respect to $\rtimes_G$.
\end{lemma}
\begin{proof}
The first statement follows from the assumption that the structure map $H\to G$ is a map of crossed $\mathcal A$-groups and the formula
\[
\mu_H(\varphi^\ast(x,y))
=(\varphi^y)^\ast(x)\varphi^\ast(y)
= \varphi^\ast(xy)
= \varphi^\ast(\mu_H(x,y))\ .
\]
The associativity and the unitality are obvious since $H(a)$ is a group for each $a\in\mathcal A$.
\end{proof}

We denote by $\mathbf{Mon}(\mathbf{Set}^{/G}_{\mathcal A},\rtimes_G)$ the category of monoid object in $\mathbf{Set}^{/G}_{\mathcal A}$ with respect to the monoidal structure $\rtimes_G$.
In particular, when $G$ is the terminal crossed $\mathcal A$-group $\mathfrak T_{\mathcal A}$, we write
\[
\mathbf{CrsMon}_{\mathcal A}:=\mathbf{Mon}(\mathbf{Set}^{/\mathfrak T_{\mathcal A}}_{\mathcal A},\rtimes_{\mathfrak T_{\mathcal A}})\ .
\]

\begin{definition}
A \emph{crossed $\mathcal A$-monoid} is just an object of $\mathbf{CrsMon}_{\mathcal A}$; i.e. a monoid object in the  category $\mathbf{Set}^{/\mathfrak T}_{\mathcal A}$ with respect to the monoidal structure $\rtimes_{\mathfrak T_{\mathcal A}}$.
We call maps in $\mathbf{CrsMon}_{\mathcal A}$ \emph{maps of crossed $\mathcal A$-monoids}.
\end{definition}

Note that since every crossed $\mathcal A$-group can be seen as one over the terminal crossed $A$-group $\mathfrak T_{\mathcal A}$, we can think of it as a crossed $\mathcal A$-monoid by virtue of \cref{lem:crsgrp-monobj}.
Hence, it makes sense to consider the slice category $\mathbf{CrsMon}^{/G}_{\mathcal A}$.
On the other hand, in view of \cref{lem:coeff-monoidal}, the map $G\to\mathfrak T_{\mathcal A}$ also induces a functor
\begin{equation}
\label{eq:mon/-crsmon}
\mathbf{Mon}(\mathbf{Set}^{/G}_{\mathcal A},\rtimes_G)
\to\mathbf{CrsMon}_{\mathcal A}\ .
\end{equation}
It immediately follows from the definition of $\rtimes_G$ that $G$ is itself a terminal object in $\mathbf{Mon}(\mathbf{Set}^{/G}_{\mathcal A},\rtimes_G)$.
Thus, the functor \eqref{eq:mon/-crsmon} factors through
\begin{equation}
\label{eq:mon/-crsmon/}
\mathbf{Mon}(\mathbf{Set}^{/G}_{\mathcal A},\rtimes_G)
\to\mathbf{CrsMon}^{/G}_{\mathcal A}\ .
\end{equation}

\begin{proposition}
\label{prop:crsmon-slice}
The functor \eqref{eq:mon/-crsmon/} is an equivalence of categories.
\end{proposition}
\begin{proof}
To see \eqref{eq:mon/-crsmon/} is essentially surjective, put $\mathfrak T_{\mathcal A}$ to be the terminal crossed $\mathcal A$-group, and observe that we have $X\rtimes_GX=X\rtimes_{\mathfrak T_{\mathcal A}}X$ as $\mathcal A$-sets for each $X\in\mathbf{Set}^{/G}_{\mathcal A}$.
Then, it turns out that a monoid structure on $X$ with respect to $\rtimes_{\mathfrak T_{\mathcal A}}$ defines one with respect to $\rtimes_G$ if and only if the map $X\to G$ is a monoid homomorphism with respect to $\rtimes_{\mathfrak T_{\mathcal A}}$.
This implies that \eqref{eq:mon/-crsmon/} is essentially surjective.
It also follows from the similar observation that \eqref{eq:mon/-crsmon/} is fully faithful.
\end{proof}

\begin{theorem}
\label{theo:crsmon-monad}
Let $\mathcal A$ be a small category and $G$ a crossed $\mathcal A$-group.
Then, the forgetful functor $\mathbf{CrsMon}^{/G}_{\mathcal A}\to\mathbf{Set}^{/G}_{\mathcal A}$ admits a left adjoint so to form a monadic adjunction:
\[
F^G_{\mathcal A}:
\vcenter{
  \xymatrix{
    \mathbf{Set}^{/G}_{\mathcal A} \ar@{}[r]|-\perp \ar@/^.2pc/[]+R+(0,1);[r]+L+(0,1) & \mathbf{CrsMon}^{/G}_{\mathcal A} \ar@/^.2pc/[]+L+(0,-1);[l]+R+(0,-1) } }
:U^G_{\mathcal A}\ .
\]
Moreover, the associated monad on $\mathbf{Set}^{/G}_{\mathcal A}$ is finitely; i.e. it commutes with filtered colimits.
Consequently, the category $\mathbf{CrsMon}^{/G}_{\mathcal A}$ is locally presentable.
\end{theorem}
\begin{proof}
According to \cref{prop:crsmon-slice}, we have an equivalence $\mathbf{CrsMon}^{/G}_{\mathcal A}\simeq\mathbf{Mon}(\mathbf{Set}^{/G}_{\mathcal A},\rtimes_G)$ respecting the functors into $\mathbf{Set}^{/G}_{\mathcal A}$.
Since the monoidal structure $\rtimes_G$ is biclosed by \cref{prop:semicrs-monoidal}, we have an explicit description for the monad $T^G_{\mathsf{mon}}$ of free monoids; namely
\[
T^G_{\mathsf{mon}}(X) = \coprod_{n=0}^\infty X^{\rtimes^{}_G n}\ ,
\]
where $X^{\rtimes_G n}$ is the $n$-fold product $X\rtimes_G\dots\rtimes_G X$.
Thus, the first statement follows.
To see $T^G_{\mathsf{mon}}$ is finitely, it suffices to show that the functor $X\mapsto X^{\rtimes_G n}$ commutes with filtered colimits for each $n\in\mathbb N$.
Now, the colimits are computed degreewisely in $\mathbf{Set}_{\mathcal A}$, and $X^{\rtimes_G n}$ is degreewisely nothing but the $n$-fold cartesian product, it follows from the same result for the category $\mathbf{Set}$.
The last statement now follows from 2.78 in \cite{AdamekRosicky1994}.
\end{proof}

\begin{corollary}
\label{cor:crsmon-locpres-w}
For every small category $\mathcal A$, the category $\mathbf{CrsMon}_{\mathcal A}$ is locally presentable.
\end{corollary}

\begin{corollary}
\label{cor:crsmon-bicomp}
The category $\mathbf{CrsMon}^{/G}_{\mathcal A}$ is complete and cocomplete.
Moreover, arbitrary limits and filtered colimits can be computed in the category $\mathbf{Set}^{/G}_{\mathcal A}$.
\end{corollary}

We want to make use of \cref{theo:crsmon-monad} to establish a required adjunction between $\mathbf{CrsGrp}^{/G}_{\mathcal A}$ and $\mathbf{Set}^{/G}_{\mathcal A}$.
In view of \cref{lem:crsgrp-monobj}, we can think of each crossed $\mathcal A$-group over $G$ as an object of $\mathbf{CrsMon}^{/G}_{\mathcal A}$.
This assignment actually defines a functor
\begin{equation}
\label{eq:crsgrp-mon-func}
\mathbf{CrsGrp}^{/G}_{\mathcal A}\to\mathbf{CrsMon}^{/G}_{\mathcal A}\ .
\end{equation}
Indeed, every map $f:H\to K$ of crossed $\mathcal A$-groups over $G$ clearly preserves the monoid structure described in \cref{lem:crsgrp-monobj}.

\begin{proposition}
\label{prop:crsgrp-ff}
The functor \eqref{eq:crsgrp-mon-func} is fully faithful.
Moreover, a crossed $\mathcal A$-monoid object $M$ belongs to the essential image if and only if it is a degreewise group.
\end{proposition}
\begin{proof}
Notice first that the following triangle is commutative:
\[
\xymatrix{
  \mathbf{CrsGrp}^{/G}_{\mathcal A} \ar[rr] \ar[dr]_{\text{forget}} && \mathbf{CrsMon}^{/G}_{\mathcal A} \ar[dl]^{\text{forget}} \\
  & \mathbf{Set}^{/G}_{\mathcal A} & }
\]
Since both of the forgetful functors are faithful, the top one is also faithful.
To see it is also full, take two crossed $\mathcal A$-groups $H$ and $K$ over $G$ and an arbitrary homomorphism $f:H\to K$ of crossed $\mathcal A$-monoids over $G$.
We show $f$ is actually a map of crossed $\mathcal A$-groups.
Since it is clearly a map of $\mathcal A$-sets that is a degreewise group homomorphism, it suffices to show $f$ respects the actions of $H(a)$ and $K(a)$ on $\mathcal A(b,a)$ for each $a,b\in\mathcal A$.
This follows from the observation that the actions of $H(a)$ and $K(a)$ factor through $G(a)$ and that we have a commutative triangle below:
\[
\xymatrix{
  H(a) \ar[rr]^f \ar[dr] && K(a) \ar[dl] \\
  & G(a) & }
\]

We finally prove the last assertion.
Let $M$ be a crossed $\mathcal A$-monoid which is a degreewise group.
Then, it is easily verified that the structure map $M\to G$ is a degreewise group homomorphism.
Hence, for each $a,b\in\mathcal A$, the group $M(a)$ inherits an action on $\mathcal A(b,a)$ from $G(a)$.
One can see this action together with the group structure make $M$ into a crossed $\mathcal A$-group.
In addition, $M\to G$ is clearly a map of crossed $\mathcal A$-groups, so that we obtain $M\in\mathbf{CrsGrp}^{/G}_{\mathcal A}$ as required.
\end{proof}

By virtue of \cref{prop:crsgrp-ff}, we may regard $\mathbf{CrsGrp}^{/G}_{\mathcal A}$ as a full subcategory of $\mathbf{CrsMon}^{/G}_{\mathcal A}$.
In fact, it is more than just a subcategory but \emph{special one}.
Namely, it is both reflective and coreflective.

We just need one lemma.

\begin{lemma}
\label{lem:crsmon-invert}
Let $M$ be a monoid object in $\mathbf{Set}^{/G}_{\mathcal A}$, and let $\varphi:a\to b\in\mathcal A$.
Then, the map $\varphi^\ast:M(b)\to M(a)$ preserves invertible elements in the monoid structures.
\end{lemma}
\begin{proof}
If $x\in M(b)$ is invertible in its monoid structure, we can describe the inverse of $\varphi^\ast(x)\in M(a)$ explicitly as follows:
\[
\varphi^\ast(x)^{-1} = (\varphi^x)^\ast(x^{-1})\ .
\]
Indeed, we have
\[
\begin{gathered}
\varphi^\ast(x)(\varphi^x)^\ast(x^{-1})
= (\varphi^{x^{-1}x})^\ast(x) (\varphi^x)^\ast(x^{-1})
= (\varphi^x)^\ast(x x^{-1})
= 1
\\
(\varphi^x)^\ast(x^{-1}) \varphi^\ast(x)
= \varphi^\ast(x x^{-1})
= 1
\end{gathered}
\]
\end{proof}

\begin{theorem}
\label{theo:crsgrp-birefl}
Let $\mathcal A$ be a small category, and let $G$ be a crossed $\mathcal A$-group.
Then, the subcategory $\mathbf{CrsGrp}^{/G}_{\mathcal A}\subset\mathbf{CrsMon}^{/G}_{\mathcal A}$ is closed under arbitrary (small) limits and colimits.
Consequently, $\mathbf{CrsGrp}^{/G}_{\mathcal A}\subset\mathbf{CrsMon}^{/G}_{\mathcal A}$ is both reflective and coreflective as a subcategory; i.e. the inclusion admits both left and right adjoints.
\end{theorem}
\begin{proof}
We first show $\mathbf{CrsGrp}^{/G}_{\mathcal A}\subset\mathbf{CrsMon}^{/G}_{\mathcal A}$ is coreflective.
The right adjoint functor $J:\mathbf{CrsMon}^{/G}_{\mathcal A}\to\mathbf{CrsGrp}^{/G}_{\mathcal A}$ is described as follows:
for a crossed $\mathcal A$-monoid $M$ over $G$, the underlying $\mathcal A$-set of $J(M)$ is degreewisely the group of invertible elements of $M(a)$, which actually forms an $\mathcal A$-subset of $M$ thanks to \cref{lem:crsmon-invert}.
It is easily verified that the composition $J(M)\hookrightarrow M\to G$ and the restricted operations exhibit $J(M)$ as a crossed $\mathcal A$-group over $G$.
Note that, for a crossed $\mathcal A$-group $H$ over $G$, each map $f:G\to H$ of crossed $\mathcal A$-monoids preserves invertible elements so it factors through $J(M)\hookrightarrow M$.
Moreover, since $J(M)\to M$ is a monomorphism, this factorization is unique.
This implies we have a natural bijection
\[
\mathbf{CrsMon}^{/G}_{\mathcal A}(H,M)
\cong\mathbf{CrsGrp}^{/G}_{\mathcal A}(H,J(M))\ .
\]
Hence, $J$ is right adjoint to the inclusion.

Next, we prove the closedness properties.
As for colimits, it follows from the coreflectivity proved above, \cref{prop:crsgrp-colim}, and \cref{cor:crsmon-bicomp}.
To see it is also the case for limits, let $H_\bullet:\mathcal I\to\mathbf{CrsGrp}^{/G}_{\mathcal A}$ be a functor from a small category $\mathcal I$.
Since the category $\mathbf{CrsMon}^{/G}_{\mathcal A}$ is complete by \cref{cor:crsmon-bicomp}, we can take the limit in the category $\mathbf{CrsMon}^{/G}_{\mathcal A}$ and write $H_\infty := \lim_{i\in\mathcal I} H_i$.
We have to show $H_\infty$ is a crossed $\mathcal A$-group.
Note that $H_\infty$ is also the limit in $\mathbf{CrsMon}^{/G}_{\mathcal A}$ of the extended diagram $H^\lhd_\bullet:\mathcal I^\lhd=\{v\}\star\mathcal I\to\mathbf{CrsGrp}^{/G}_{\mathcal A}$ given by
\[
H^\lhd_x=
\begin{cases}
H_x \quad& x\in\mathcal I \\
G\quad & x=v\ .
\end{cases}
\]
In view of \cref{cor:crsmon-bicomp}, limits of connected diagrams in $\mathbf{CrsMon}^{/G}_{\mathcal A}$ is computed degreewisely, so for each $a\in\mathcal A$, we have
\[
H_\infty(a)\cong \lim_{x\in\mathcal I^\lhd}H^\lhd_x(a)\ .
\]
The right hand side is a limit of groups, so it is again a group.
It follows that $H_\infty$ is a degreewise group so it belongs to $\mathbf{CrsGrp}^{/G}_{\mathcal A}$ by virtue of \cref{prop:crsgrp-ff}.

To see the last statement, we have to show the embedding $\mathbf{CrsGrp}^{/G}_{\mathcal A}\to\mathbf{CrsMon}^{/G}_{\mathcal A}$ admits a left adjoint.
This follows from the first assertion, \cref{prop:crsgrp-colim}, and the \emph{General Adjoint Functor Theorem}.
\end{proof}

In summary, we obtain the following adjunctions:
\[
\xymatrix{
  \mathbf{CrsGrp}_{\mathcal A} \ar@{^(->}[]+R+(1,0);[r] \ar@{}[r]^-\perp_-\perp & \mathbf{CrsMon}_{\mathcal A} \ar@/^.3pc/[]+DL;[l]+DR \ar@/_.3pc/[]+UL;[l]+UR \ar@/^.2pc/[]+R+(0,1);[r]+L+(0,1) \ar@{}[r]|-\perp & \mathbf{Set}_{\mathcal A/S} \ar@/^.2pc/[]+L+(0,-1);[l]+R+(0,-1) }
\]
where the right one is monadic, and each arrow in right direction creates arbitrary limits and filtered colimits.

\begin{remark}
If one knows a terminal crossed $\mathcal A$-group $\mathfrak T_{\mathcal A}$, then he can prove \cref{theo:crsgrp-birefl} in a more concrete way.
Indeed, it is a consequence of \cref{theo:crsmon-monad}, \cref{theo:crsgrp-birefl} above, and Corollary 2.4 in \cite{MakkaiPitts1987}.
\end{remark}

\section{Base change of crossed monoids}
\label{sec:basechange}

It is often the case that the category $\mathcal A$ is related to another category, say $\widetilde{\mathcal A}$, by a functor $\varphi:\mathcal A\to\widetilde{\mathcal A}$ in nature.
Such a functor gives rise to adjunctions
\[
\varphi_!,\varphi_\ast:
\vcenter{
  \xymatrix{
    \mathbf{Set}_{\mathcal A} \ar[r]^-\perp_-\perp \ar@/^.3pc/[]+UR;[r]+UL \ar@/_.3pc/[]+DR;[r]+DL & \mathbf{Set}_{\widetilde{\mathcal A}} }}
:\varphi^\ast\ ,
\]
where $\varphi_!$ and $\varphi_\ast$ are the left and right Kan extensions of $\varphi$ respectively along the Yoneda embedding.
More generally, for an $\widetilde{\mathcal A}$-set $\widetilde S$, the functor $\varphi^\ast$ induces a functor
\[
\varphi^\ast_{\widetilde S}:\mathbf{Set}_{\widetilde{\mathcal A}}^{/\widetilde S}\to\mathbf{Set}_{\mathcal A}^{/\varphi^\ast\widetilde S}\ ,
\]
which also has both left and right adjoint functors; namely, define $\varphi^{\widetilde S}_!,\varphi^{\widetilde S}_\ast:\mathbf{Set}_{\mathcal A}^{/\varphi^\ast\widetilde S}\to\mathbf{Set}_{\widetilde{\mathcal A}}^{/\widetilde S}$ as follows: for $X\in\mathbf{Set}_{\mathcal A}^{/\varphi^\ast\widetilde S}$, $\varphi^{\widetilde S}_!X:=\varphi_!X$ with the adjoint morphism $\varphi_!X\to S$, and $\varphi^{\widetilde S}_\ast X$ is the object in the pullback square below:
\[
\xymatrix{
  \varphi^{\widetilde S}_\ast X \ar[d] \ar[r] \ar@{}[dr]|(.4)\pbcorner & \varphi_\ast X \ar[d] \\
  \widetilde S \ar[r] & \varphi_\ast \varphi^\ast\widetilde S }
\]
Then, we obtain adjunctions
\[
\varphi^{\widetilde S}_!,\varphi^{\widetilde S}_\ast:
\vcenter{
  \xymatrix{
    \mathbf{Set}_{\mathcal A}^{/\varphi^\ast\widetilde S} \ar@/^.4pc/[]+R+(0,2);[r]+L+(0,2) \ar@/_.4pc/[]+R+(0,-2);[r]+L+(0,-2) & \mathbf{Set}_{\widetilde{\mathcal A}}^{/\widetilde S} \ar[l]^-\perp_-\perp }}
:\varphi^\ast_{\widetilde S}\ ,
\]
The lifts of these functors to the categories of crossed monoids are the central interest in this section.

Now, fix a functor $\varphi:\mathcal A\to\widetilde{\mathcal A}$, and let $\widetilde G$ be a crossed $\widetilde{\mathcal A}$-group.
The notion of crossed monoids anyway arises from the monoidal structure on the category of presheaves as mentioned in \cref{sec:monobj}, we investigate how the functor $\varphi^\ast_{\widetilde G}$ above relates monoidal structures.
At the first glance, one can notice that, for this question to make sense, we have to give $\varphi^\ast\widetilde G$ a structure of $\mathcal A$-groups.
Unfortunately, it immediately turns out that there is no canonical way to do this, so we give up the general cases and concentrate only on \emph{faithful} functors.

\begin{definition}
Let $\varphi:\mathcal A\to\widetilde{\mathcal A}$ be a faithful functor between small categories.
Then, for a crossed $\widetilde{\mathcal A}$-group $\widetilde G$, $\varphi$ is said to be \emph{$\widetilde G$-stable} if for each $a,b\in\mathcal A$, the image of the map
\[
\varphi:\mathcal A(a,b)\to\widetilde{\mathcal A}(\varphi(a),\varphi(b))
\]
is $\widetilde G(b)$-stable.
\end{definition}

For example, fully faithful functors are stable for every crossed group.
The following result is straightforward.

\begin{lemma}
\label{lem:pb-crsgrp}
Let $\varphi:\mathcal A\to\widetilde{\mathcal A}$ be a faithful functor between small categories.
Suppose $\widetilde G$ is a crossed $\widetilde{\mathcal A}$-group such that $\varphi$ is $\widetilde G$-stable.
Then, the $\mathcal A$-set $\varphi^\ast\widetilde G$ admits a unique structure of crossed $\mathcal A$-groups such that
\begin{enumerate}[label={\rm(\roman*)}]
  \item for each $a\in\mathcal A$, the group structure on $\varphi^\ast\widetilde G(a)=\widetilde G(\varphi(a))$ agrees with the original one;
  \item for each $a,b\in\mathcal A$, the map
\[
\varphi:\mathcal A(a,b)\to\widetilde{\mathcal A}(\varphi(a),\varphi(b))
\]
is $\varphi^\ast\widetilde G(b)$-equivariant.
\end{enumerate}
\end{lemma}

In the following, we always regard $\varphi^\ast\widetilde G$ as a crossed $\mathcal A$-group with the structure in \cref{lem:pb-crsgrp} whenever $\varphi$ is $\widetilde G$-stable faithful functor.
Hence, the monoidal structure $\rtimes_{\varphi^\ast\widetilde G}$ on the category $\mathbf{Set}_{\mathcal A}^{/\varphi^\ast\widetilde G}$ makes sense (see \cref{sec:monobj}).

\begin{proposition}
\label{prop:pb-monoidal}
Let $\varphi:\mathcal A\to\widetilde{\mathcal A}$ be a faithful functor between small categories which is $\widetilde G$-stable for a crossed $\widetilde{\mathcal A}$-group $\widetilde G$.
Then, the induced functor
\[
\varphi^\ast_{\widetilde G}:\mathbf{Set}_{\widetilde{\mathcal A}}^{/\widetilde G}\to\mathbf{Set}_{\mathcal A}^{/\varphi^\ast\widetilde G}
\]
is monoidal with respect to the monoidal structures $\rtimes_{\widetilde G}$ and $\rtimes_{\varphi^\ast\widetilde G}$.
\end{proposition}
\begin{proof}
Note that, for each $\widetilde X,\widetilde Y\in\mathbf{Set}_{\widetilde{\mathcal A}}^{/\widetilde G}$, and for each $a\in\mathcal A$, we have a canonical identification
\[
\begin{multlined}
\varphi^\ast(\widetilde X\rtimes_{\widetilde G}\widetilde Y)(a)
= (\widetilde X\rtimes_{\widetilde G}\widetilde Y)(\varphi(a))
= \widetilde X(\varphi(a))\times \widetilde Y(\varphi(a)) \\[1ex]
= \varphi^\ast\widetilde X(a)\times \varphi^\ast\widetilde Y(a)
= (\varphi^\ast\widetilde\rtimes_{\varphi^\ast\widetilde G}\varphi^\ast\widetilde Y)(a)\ .
\end{multlined}
\]
On the other hand, for each morphism $\varphi:a\to b\in\mathcal A$, the induced map
\[
\varphi^\ast:\widetilde X(\varphi(b))\times\widetilde Y(\varphi(b))
\to\widetilde X(\varphi(a))\times\widetilde Y(\varphi(a))
\]
is, no matter whether it is considered in $\varphi^\ast(\widetilde X\rtimes_{\widetilde G}\widetilde Y)$ or $\varphi^\ast\widetilde X\rtimes_{\varphi^\ast\widetilde G}\varphi^\ast\widetilde Y$, given by
\[
\varphi^\ast(x,y)
= \varphi(\varphi)^\ast(x,y)
= ((\varphi(\varphi)^y)^\ast(x),\varphi(\varphi)^\ast(y))
= (\varphi(\varphi^y)^\ast(x),\varphi(\varphi)^\ast(y))\ .
\]
Thus, we obtain a canonical identification $\varphi^\ast(\widetilde X\rtimes_{\widetilde G}\widetilde Y)=\varphi^\ast\widetilde X\rtimes_{\varphi^\ast \widetilde G}\varphi^\ast\widetilde Y$, which makes $\varphi$ into a monoidal functor.
\end{proof}

\begin{corollary}
\label{cor:Ran-crsmon}
Let $\varphi:\mathcal A\to\widetilde{\mathcal A}$ and $\widetilde G\in\mathbf{CrsGrp}_{\widetilde{\mathcal A}}$ be as in \cref{prop:pb-monoidal}.
Then, the adjunction $\varphi^\ast_{\widetilde G}\dashv\varphi^{\widetilde G}_\ast$ induces an adjunction
\[
\varphi^\natural_{\widetilde G}:
\vcenter{
  \xymatrix{
    \mathbf{CrsMon}^{/\widetilde G}_{\widetilde{\mathcal A}} \ar@{}[r]|-\perp \ar@/^.2pc/[]+R+(0,1);[r]+L+(0,1) & \mathbf{CrsMon}^{/\varphi^\ast\widetilde G}_{\mathcal A} \ar@/^.2pc/[]+L+(0,-1);[l]+R+(0,-1) } }
:\varphi^{\widetilde G}_\sharp\ .
\]
\end{corollary}
\begin{proof}
The statement is a consequence of \cref{prop:pb-monoidal} and the fact that, for a monoidal functor $F:\mathcal C\to\mathcal D$, the induced functor $\mathbf{Mon}(\mathcal C)\to\mathbf{Mon}(\mathcal D)$ admits a right adjoint as soon as so does $F$.
The reader will find the full proof of this in the section 2.3 in \cite{Porst2008}.
\end{proof}

Note that the fact we used in the proof of \cref{cor:Ran-crsmon} not only shows the existence of the adjunction but also provides a way to compute it.
Indeed, if $\widetilde M$ is a crossed $\widetilde{\mathcal A}$-monoid over $\widetilde G$, then $\varphi^\ast_{\widetilde G}\widetilde M$ has a canonical structure of crossed $\mathcal A$-monoids in view of \cref{prop:pb-monoidal}.
On the other hand, since the functor $\varphi^{\widetilde G}_\ast$ is right adjoint to the monoidal functor $\varphi^\ast_{\widetilde G}$, it admits a structure of lax monoidal functors; namely we have $\varphi^{\widetilde G}_\ast(\ast)\cong\ast$ and a natural transformation
\[
\mu:\varphi^{\widetilde G} X\rtimes_{\widetilde G}\varphi^{\widetilde G} Y
\to \varphi^{\widetilde G}_\ast(X\rtimes_{\varphi^\ast\widetilde G}Y)
\]
subject to an appropriate coherence conditions.
Then, for each crossed $\mathcal A$-monoid $M$ over $\varphi^\ast\widetilde G$, $\varphi^{\widetilde G}_\ast M$ admits a canonical structure of crossed $\widetilde{\mathcal A}$-monoids over $\widetilde G$ as
\[
\begin{gathered}
\varphi^{\widetilde G}_\ast M\rtimes_{\widetilde G}\varphi^{\widetilde G}_\ast M
\xrightarrow\mu \varphi^{\widetilde G}_\ast(M\rtimes_{\varphi^\ast\widetilde G} M)
\to\varphi^{\widetilde G}_\ast M \\
\ast\cong\varphi^{\widetilde G}_\ast(\ast)\to \varphi^{\widetilde G}_\ast M\ .
\end{gathered}
\]
This actually gives the right adjoint $\varphi^{\widetilde G}_\sharp$ in the adjunction in \cref{cor:Ran-crsmon}.

\begin{proposition}
\label{prop:Ranff-crsmon}
Let $\varphi:\mathcal A\to\widetilde{\mathcal A}$ be a fully faithful functor, and let $\widetilde G$ be a crossed $\widetilde{\mathcal A}$-group.
Then, the functor $\varphi^{\widetilde G}_\sharp:\mathbf{CrsMon}^{/\varphi^\ast\widetilde G}_{\mathcal A}\to\mathbf{CrsMon}^{/\widetilde G}_{\widetilde{\mathcal A}}$ induced by the right Kan extension of $\varphi$ is fully faithful.
\end{proposition}
\begin{proof}
Since $\varphi$ is fully faithful, the induced functor
\[
\varphi_{/\widetilde G}:\mathcal A/(\varphi^\ast\widetilde G)\to\widetilde{\mathcal A}/\widetilde G
\]
is also fully faithful.
Note that we have canonical identifications
\[
\mathbf{Set}^{/\varphi^\ast\widetilde G}_{\mathcal A}
\simeq \mathbf{Set}_{\mathcal A/\varphi^\ast\widetilde G}
\quad\text{and}\quad
\mathbf{Set}^{/\widetilde G}_{\widetilde{\mathcal A}}
\simeq \mathbf{Set}_{\widetilde{\mathcal A}/\widetilde G}
\]
so that the adjunction $\varphi^\ast_{\widetilde G}\dashv\varphi^{\widetilde G}_\ast$ is identified with the one obtained by the right Kan extension of $\varphi_{/\widetilde G}$.
Since Kan extensions of fully faithful functors along the Yoneda embedding are again fully faithful, e.g see Proposition 4.23 in \cite{Kel05}, $\varphi^{\widetilde G}_\ast$ is fully faithful.
The last assertion is equivalent to that, for each $X\in\mathbf{Set}^{/\varphi^\ast\widetilde G}_{\mathcal A}$, the counit
\begin{equation}
\label{eq:prf:Ranff-crsmon:counit}
\varphi^\ast_{\widetilde G}\varphi^{\widetilde G}_\ast X\to X
\end{equation}
is an isomorphism.
Note that if $X$ is a crossed $\mathcal A$-monoid over $\varphi^\ast\widetilde G$, then \eqref{eq:prf:Ranff-crsmon:counit} underlies the counit map $\varphi^\natural_{\widetilde G}\varphi^{\widetilde G}_\sharp X\to X$.
This implies that the right adjoint $\varphi^{\widetilde G}_\sharp$ is fully faithful.
\end{proof}

\begin{example}
\label{ex:Ran-aug}
Take $\mathcal A\to\widetilde{\mathcal A}$ to be the inclusion $j:\Delta\hookrightarrow\widetilde\Delta$.
The right adjoint functor $j_\ast$ to the restriction $j^\ast:\mathbf{Set}_{\widetilde\Delta}\to\mathbf{Set}_{\Delta}$ sends each simplicial set $X_\bullet$ to the augmented one $j_\ast X$ given by
\begin{equation}
\label{eq:ex:Ran-aug:func}
j_\ast X(\langle n\rangle)=
\begin{cases}
\{\mathrm{pt}\} \quad & n=0 \\
X_{n-1} \quad & n\ge 1\ .
\end{cases}
\end{equation}
In particular, we have a canonical isomorphisms
\begin{equation}
\label{eq:ex:Ran-aug:termcrs}
j^\ast\mathfrak W_{\widetilde\Delta} \cong \mathfrak W_\Delta
\ ,\quad j_\ast\mathfrak W_\Delta\cong\mathfrak W_{\widetilde\Delta}\ .
\end{equation}
Now, since the functor $j$ is fully faithful, it is $\mathfrak W_{\widetilde\Delta}$-stable, so we obtain an adjunction
\[
j^\ast_{\mathfrak W_{\widetilde\Delta}}:
\vcenter{
  \xymatrix{
    \mathbf{Set}^{/\mathfrak W_{\widetilde\Delta}}_{\widetilde\Delta} \ar@{}[r]|-\perp \ar@/^.2pc/[]+R+(0,1);[r]+L+(0,1) & \mathbf{Set}^{/\mathfrak W_\Delta}_\Delta \ar@/^.2pc/[]+L+(0,-1);[l]+R+(0,-1) }}
:j_\ast^{\mathfrak W_{\widetilde\Delta}}
\]
with $j^\ast_{\mathfrak W_{\widetilde\Delta}}$ monoidal by \cref{prop:pb-monoidal}.
We finally obtain an adjunction
\[
j^\natural:
\vcenter{
  \xymatrix{
    \mathbf{CrsMon}_{\widetilde\Delta} \ar@{}[r]|-\perp \ar@/^.2pc/[]+R+(0,1);[r]+L+(0,1) & \mathbf{CrsMon}_\Delta \ar@/^.2pc/[]+L+(0,-1);[l]+R+(0,-1) }}
:j_\sharp\ .
\]
Note that, thanks to the equation \eqref{eq:ex:Ran-aug:termcrs}, the functor $j_\sharp$ is also given by \eqref{eq:ex:Ran-aug:func} and fully faithful by virtue of \cref{prop:Ranff-crsmon}.
\end{example}

\begin{example}
\label{ex:Ran-interval}
By its construction, the category $\nabla$ is the Kleisli category of the monad
\[
\mathfrak J:\widetilde\Delta\to\widetilde\Delta
\ ;\quad \langle n\rangle\mapsto \langle n+2\rangle\cong\{-\infty\}\star\langle n\rangle\star\{\infty\}\cong\double\langle n\double\rangle\ .
\]
In view of this, one can find the right adjoint to the canonical embedding $\mathfrak J:\widetilde\Delta\to\nabla$; namely
\[
U:\nabla\to\widetilde\Delta\ ;\quad \double\langle n\double\rangle\to\langle n+2\rangle\ .
\]
It turns out that the pullbacks along these functors gives rise to an adjunction $\mathfrak J^\ast\dashv U^\ast:\mathbf{Set}_{\widetilde\Delta}\to\mathbf{Set}_\nabla$, and the uniqueness of right adjoints implies $U^\ast\cong \mathfrak J_\ast$.
More explicitly, the right adjoint $\mathfrak J_\ast:\mathbf{Set}_{\widetilde\Delta}\to\mathbf{Set}_\nabla$ is given by
\[
\mathfrak J_\ast X(\double\langle n\double\rangle) = X(\langle n+2\rangle)\ .
\]
The counit $X\to\mathfrak J_\ast\mathfrak J^\ast X$ is the monomorphism described as follows: consider the map $\tau_n:\double\langle n+2\double\rangle\to\double\langle n\double\rangle$ defined by
\[
\tau_n(i) =
\begin{cases}
\hfil -\infty\hfil & i=-\infty,1 \\
\hfil i-1\hfil & 2\le i\le n+1 \\
\hfil \infty\hfil & i=n+2,\infty\ .
\end{cases}
\]
Then, each component $X\to\mathfrak J_\ast\mathfrak J^\ast X$ is the map induced by $\tau_n$:
\[
\tau_n^\ast:X(\double\langle n\double\rangle)
\to\mathfrak J_\ast\mathfrak J^\ast X(\double\langle n\double\rangle)
=X(\double\langle n+2\double\rangle)\ .
\]
In particular, as for the Weyl crossed interval group $\mathfrak W_\nabla$, the unit map $\mathfrak W_\nabla\to \mathfrak J_\ast \mathfrak J^\ast\mathfrak W_\nabla$ exhibits $\mathfrak W_\nabla(\double\langle n\double\rangle)$ as a subset of $\mathfrak W_\nabla(\double\langle n+2\double\rangle)$ given by
\[
\left\{(\sigma;\varepsilon_1,\dots,\varepsilon_{n+2};\theta)\in\mathfrak W_\nabla(\double\langle n+2\double\rangle)\ \middle|\
\substack{\displaystyle\sigma(\{1,n+2\})=\{1,n+2\}\ ,
  \cr
  \displaystyle\left.\sigma\right|_{\{1,n+2\}}=\varepsilon_1=\varepsilon_{n+2}=\theta}\right\}\ .
\]
On the other hand, we have
\[
\mathfrak J^\ast\mathfrak W_\nabla
\cong\mathfrak H\times C_2
\cong\mathfrak W_{\widetilde\Delta}\times C_2\ .
\]
Since $\mathfrak W_{\widetilde\Delta}$ is the terminal object in $\mathbf{CrsMon}_{\widetilde\Delta}$, giving a map $M\to\mathfrak J^\ast\mathfrak W_\nabla$ of augmented crossed simplicial monoid is equivalent to giving an augmented simplicial map $M\to C_2$ which is a degreewise monoid homomorphism.
Hence, for $M\in\mathbf{CrsMon}^{/\mathfrak W_{\widetilde\Delta}\times C_2}_{\widetilde\Delta}$ with associated map $\theta:M\to C_2$, we have
\[
\mathfrak J^{\mathfrak W_\nabla}_\sharp M(\double\langle n\double\rangle)
= \left\{x\in M(\langle n+2\rangle)\ \middle|\
\substack{\displaystyle x(\{1,n+2\})=\{1,n+2\},\cr
  \displaystyle\left.x\right|_{\{1,n+2\}}=\varepsilon_1(x)=\varepsilon_{n+2}(x)=\theta(x)}\right\}\ ,
\]
where we write $(x;\varepsilon_1(x),\dots,\varepsilon_{n+2}(x))$ the image of $x$ in $\mathfrak W_{\widetilde\Delta}$.
It exactly gives the right adjoint in the adjunction
\begin{equation}
\label{eq:Ran-interval:adjcrs}
\mathfrak J^\natural_{\mathfrak W_\nabla}:
\vcenter{
  \xymatrix{
    \mathbf{CrsMon}_\nabla \ar@{}[r]|-\perp \ar@/^.2pc/[]+R+(0,1);[r]+L+(0,1) & \mathbf{CrsMon}_{\widetilde\Delta}^{/\mathcal W_{\widetilde\Delta}\times C_2} \ar@/^.2pc/[]+L+(0,-1);[l]+R+(0,-1) }}
:\mathfrak J_\sharp^{\mathfrak W_\nabla}\ .
\end{equation}
Note that this adjunction extends to the right: namely, the projection $\mathfrak W_{\widetilde\Delta}\times C_2\to\mathfrak W_{\widetilde\Delta}$, which is the unique map of augmented crossed simplicial groups between them, gives rise to an adjunction
\begin{equation}
\label{eq:Ran-interval:C2}
\begin{array}{rcl}
  \mathbf{CrsMon}_{\widetilde\Delta}^{/\mathfrak W_{\widetilde\Delta}\times C_2} & \vcenter{\xymatrix{ *{} \ar@{}[r]|-\perp \ar@/^.2pc/[]+R+(0,1);[r]+L+(0,1) & *{} \ar@/^.2pc/[]+L+(0,-1);[l]+R+(0,-1) }} & \mathbf{CrsMon}_{\widetilde\Delta} \\[2ex]
  \left(M\to\mathfrak W_{\widetilde\Delta}\times C_2\right) & \xymatrix{ *{} \ar@{|->}[r] & *{} } & M \\[2ex]
  \left(M\times C_2\to\mathfrak W_{\widetilde\Delta}\times C_2\right) & \xymatrix{ *{} & *{} \ar@{|->}[l] } & M
\end{array}
\end{equation}
Combining \eqref{eq:Ran-interval:adjcrs} with \eqref{eq:Ran-interval:C2}, we obtain an adjunction
\[
\mathfrak J^\natural:
\vcenter{
  \xymatrix{
    \mathbf{CrsMon}_\nabla \ar@{}[r]|-\perp \ar@/^.2pc/[]+R+(0,1);[r]+L+(0,1) & \mathbf{CrsMon}_{\widetilde\Delta} \ar@/^.2pc/[]+L+(0,-1);[l]+R+(0,-1) }}
:\mathfrak J_\sharp\ .
\]
\end{example}

We next discuss the other adjunction $\varphi^{\widetilde G}_!\dashv\varphi^\ast_{\widetilde G}$ induced by the left Kan extension of a $\widetilde G$-stable faithful functor $\varphi:\mathcal A\to\widetilde{\mathcal A}$.
In contrast to the right Kan extension, the functor $\varphi^{\widetilde G}_!$ itself does not induce any functor on the category of monoids in general.
We hence need to begin with the construction of the left adjoint to the functor
\[
\varphi^\natural_{\widetilde G}:\mathbf{CrsMon}^{/\widetilde G}_{\widetilde{\mathcal A}}
\to\mathbf{CrsMon}^{/\varphi^\ast\widetilde G}_{\mathcal A}
\]
directly.
Fortunately, we can make use of the following theorem.

\begin{theorem}[Adjoint Lifting Theorem, Theorem 4.5.6 in \cite{Bor94}]
\label{theo:adj-lift}
Suppose we have a square of functors
\[
\xymatrix{
  \mathcal M \ar[r]^Q \ar[d]_U \ar@{}[dr]|-\cong & \mathcal N \ar[d]^V \\
  \mathcal C \ar[r]^R & \mathcal D }
\]
which is commutative up to a natural isomorphism, and suppose $\mathcal M$ has all coequalizers.
Then, $Q$ has a left adjoint as soon as so does $R$.
More precisely, if $F\dashv U$, $G\dashv V$, and $L\dashv R$, so $FL\dashv RU\cong VQ$, then the left adjoint $K:\mathcal N\to\mathcal M$ to $Q$ is defined by the coequalizer sequence
\begin{equation}
\label{eq:adj-lift:coeq}
FLVGV(N)
\xrightrightarrows{FLV\varepsilon,\ \omega V} FLV(N)
\to K(N)\ ,
\end{equation}
in $\mathcal M$, where $\varepsilon:GV\to\mathrm{Id}$ is the counit of the adjunction, and $\omega$ is the composition
\[
\begin{gathered}
\alpha:G
\xrightarrow{G\eta^{FL\dashv VQ}} GVQFL
\xrightarrow{\varepsilon QFL} QFL
\\
\omega:FLVG
\xrightarrow{FLV\alpha} FLVQFL
\xrightarrow{\varepsilon^{FL\dashv VQ}FL} FL\ .
\end{gathered}
\]
\end{theorem}

\begin{remark}
One can deduce \Cref{theo:adj-lift} from, besides the direct proof, a more general theorem called Adjoint Triangle Theorem.
We refer the reader to \cite{Dubuc1968} and \cite{StreetVerity2010}.
\end{remark}

We can apply \cref{theo:adj-lift} in our situation anyway, i.e. to the square
\begin{equation}
\label{eq:crsmon-adjlift}
\vcenter{
  \xymatrix{
    \mathbf{CrsMon}^{/\widetilde G}_{\widetilde{\mathcal A}} \ar[r]^{\varphi^\natural_{\widetilde G}} \ar[d]_{U^{\widetilde G}} & \mathbf{CrsMon}^{/\varphi^\ast\widetilde G}_{\mathcal A} \ar[d]^{U^{\varphi^\ast\widetilde G}} \\
    \mathbf{Set}^{/\widetilde G}_{\widetilde{\mathcal A}} \ar[r]^{\varphi^\ast_{\widetilde G}} & \mathbf{Set}^{/\varphi^\ast\widetilde G}_{\mathcal A} }}
\quad.
\end{equation}
To obtain a more explicit description, however, we need to know more about each involved functor.
In particular, since all the right adjoints just forget structures, it is enough to care about the left adjoints.
We first look at the free functor
\[
F^G:\mathbf{Set}^{/G}_{\mathcal A}\to\mathbf{CrsMon}^{/G}_{\mathcal A}
\]
for a crossed $\mathcal A$-group $G$.
This functor is the one appeared in \cref{theo:crsmon-monad} and computed as follows: for each $X\in\mathbf{Set}^{/G}_{\mathcal A}$ with the structure map $p:X\to G$, $F^GX$ is, as an $\mathcal A$-set, degreewisely the free monoid generated by $X$ with the structure map
\[
\begin{array}{rccl}
  \varphi^\ast:& F^GX(b) & \to & F^GX(a) \\[1ex]
  & x_1 x_2\dots x_n & \mapsto & (\varphi^{p(x_2)\dots p(x_n)})^\ast(x_1)(\varphi^{p(x_3)\dots p(x_n)})^\ast(x_2)\dots\varphi^\ast(x_n)\ .
\end{array}
\]
for each $\varphi:a\to b\in\mathcal A$.
The map $F^GX\to G$ is the induced one.

On the other hand, for a functor $\varphi:\mathcal A\to\widetilde{\mathcal A}$, its left Kan extension $\varphi_!:\mathbf{Set}_{\mathcal A}\to\mathbf{Set}_{\widetilde{\mathcal A}}$ along the Yoneda embedding is realized as follows: for $X\in\mathbf{Set}_{\mathcal A}$ and for $\tilde a\in\widetilde{\mathcal A}$, $\varphi_!X(\tilde a)$ is the quotient set
\[
\left\{(x,\widetilde\varphi)\ \middle|\ x\in X(b),\ \widetilde\varphi\in\widetilde{\mathcal A}(\tilde a,\varphi(b))\ \text{for}\ b\in\mathcal A\right\}\big/\sim
\]
by the equivalence relation $\sim$ generated by
\[
(\theta^\ast(x),\widetilde\varphi)\sim (x,\varphi(\theta)\widetilde\varphi)
\]
for each triples $(x,\widetilde\varphi,\theta)$ such that both sides make sense.
We write $[x,\widetilde\varphi]\in\varphi_!X(\tilde a)$ the equivalence class represented by the pair $(x,\widetilde\varphi)$.
If $X$ is equipped with an $\mathcal A$-map $f:X\to\varphi^\ast\widetilde S$, then we have an $\widetilde{\mathcal A}$-map
\[
\varphi_!X\to\widetilde S
\ ;\quad [x,\widetilde\varphi] \mapsto \widetilde\varphi^\ast(f(x))\ ,
\]
which exactly gives the functor $\varphi^{\widetilde S}_!:\mathbf{Set}^{/\varphi^\ast\widetilde S}_{\mathcal A}\to\mathbf{Set}^{/\widetilde S}_{\widetilde{\mathcal A}}$.

Combining the observations above and \cref{theo:adj-lift}, we obtain the following result.

\begin{theorem}
\label{theo:Lan-crsmon}
Let $\widetilde G$ be a crossed $\widetilde{\mathcal A}$-group, and let $\varphi:\mathcal A\to\widetilde{\mathcal A}$ be a $\widetilde G$-stable faithful functor.
Then, the pullback $\varphi^\natural_{\widetilde G}$ admits a left adjoint functor $\varphi^{\widetilde G}_\flat$ so to form an adjunction
\[
\varphi^{\widetilde G}_\flat:
\vcenter{
  \xymatrix{
    \mathbf{CrsMon}^{/\varphi^\ast\widetilde G}_{\mathcal A} \ar@{}[r]|-\perp \ar@/^.2pc/[]+R+(0,1);[r]+L+(0,1) & \mathbf{CrsMon}^{/\widetilde G}_{\widetilde{\mathcal A}} \ar@/^.2pc/[]+L+(0,-1);[l]+R+(0,-1) }}
:\varphi^\natural_{\widetilde G}\ .
\]
More precisely, for each $M\in\mathbf{CrsMon}^{/\varphi^\ast\widetilde G}_{\mathcal A}$ with the structure map $p:M\to\varphi^\ast\widetilde G$, the crossed $\widetilde{\mathcal A}$-monoid $\varphi^{\widetilde G}_\flat M$ over $\widetilde G$ is given as follows: for each $\tilde a\in\widetilde{\mathcal A}$, the monoid $\varphi^{\widetilde G}_\flat M(\tilde a)$ is obtained as the quotient of the free monoid with generating set
\begin{equation}
\label{eq:Lan-quotient}
\left\{(x,\widetilde\varphi)\ \middle|\ x\in M(b),\ \widetilde\varphi\in\widetilde{\mathcal A}(\tilde a,\varphi(b))\ \text{for}\ b\in\mathcal A\right\}
\end{equation}
by the congruence relation $\sim$ generated by
\[
(e_b,\widetilde\varphi)
\sim e_a
\ ,\quad
(xy,\widetilde\varphi)
\sim(x,\widetilde\varphi^y)(y,\widetilde\varphi)
\ ,\quad
(\theta^\ast(z),\widetilde\varphi)
\sim(z,\varphi(\theta)\widetilde\varphi)
\]
for $x,y,z,\widetilde\varphi,\theta$ such that each sides makes sense.
For each $\widetilde\psi:\widetilde b\to\tilde a\in\widetilde{\mathcal A}$, the map $\widetilde\psi^\ast:\varphi^{\widetilde G}_\flat M(\tilde a)\to\varphi^{\widetilde G}_\flat M(\widetilde b)$ is given by
\[
\widetilde\psi([x_1,\widetilde\varphi_1]\dots[x_n,\widetilde\varphi_n])
= [x_1,\widetilde\psi^{\widetilde\varphi_2^\ast(p(x_2))\dots\widetilde\varphi_n^\ast(p(x_n))}\widetilde\varphi_1]\dots[x_n,\widetilde\psi\widetilde\varphi_n]\ .
\]
Finally, $\varphi^{\widetilde G}_\flat M\to\widetilde G$ is the one generated by $[x,\widetilde\varphi]\mapsto\widetilde\varphi^\ast(p(x))$.
\end{theorem}
\begin{proof}
We apply \cref{theo:adj-lift} to the diagram \eqref{eq:crsmon-adjlift}.
To simplify the notation, we write $G:=\varphi^\ast\widetilde G$ and omit all the forgetful functors from formulas.
Then, the first thing we need is to know the two morphisms
\begin{equation}
\label{eq:prf:Lan-crsmon:cofork}
F^{\widetilde G}\varphi^{\widetilde G}_!\varepsilon,\ \omega:
F^{\widetilde G}\varphi^{\widetilde G}_!F^GM
\rightrightarrows F^{\widetilde G}\varphi^{\widetilde G}_!M
\end{equation}
of \eqref{eq:adj-lift:coeq} for each $M\in\mathbf{CrsGrp}^{/G}_{\mathcal A}$.
According to the discussion above, for each $\tilde a\in\widetilde{\mathcal A}$, the elements of $F^{\widetilde G}\varphi^{\widetilde G}_!F^GM(\tilde a)$ are finite words in the quotient set
\[
\left\{(x_1,\dots,x_n;\widetilde\varphi)
\ \middle|\ n\in\mathbb N,\ x_i\in M(b),\ \widetilde\varphi\in\widetilde{\mathcal A}(\tilde a,\varphi(b))\ \text{for}\ b\in\mathcal A\right\}
\big/ \sim
\]
by the equivalence relation $\sim$ generated by
\[
(x_1,\dots,x_n;\varphi(\theta)\widetilde\varphi)
\sim ((\theta^{x_2\dots x_n})^\ast(x_1),\dots,\theta^\ast(x_n);\widetilde\varphi)\ .
\]
We write $[x_1,\dots,x_n;\widetilde\varphi]$ the equivalence class represented by the tuple $(x_1,\dots,x_n;\widetilde\varphi)$.
On the other hand, $F^{\widetilde G}\varphi^{\widetilde G}_!M(\tilde a)$ is the set of words in the set $\varphi^{\widetilde G}_!M(\tilde a)$, which is obtained as the quotient of the set \eqref{eq:Lan-quotient} as mentioned just before \cref{theo:Lan-crsmon}.
Then, the direct computation shows the two maps in \eqref{eq:prf:Lan-crsmon:cofork} are the monoid homomorphisms generated by the maps
\[
\begin{multlined}
[x_1,\dots,x_n;\widetilde\varphi] \\
\mapsto [x_1\dots x_n,\widetilde\varphi],
\ [x_1,\widetilde\varphi^{x_2\dots x_n}]\dots[x_n,\widetilde\varphi]\ ,
\end{multlined}
\]
where one multiplies $x_1,\dots,x_n$ in $M$ while the other distributes the brackets.
Therefore one obtains the required presentation.
\end{proof}

\begin{corollary}
\label{cor:Lanff-crsmon}
Let $\varphi:\mathcal A\to\widetilde{\mathcal A}$ be a fully faithful functor, and let $\widetilde G$ be a crossed $\widetilde{\mathcal A}$-group.
Then, the left adjoint functor $\varphi^{\widetilde G}_\flat:\mathbf{CrsMon}^{/\varphi^\ast\widetilde G}_{\mathcal A}\to\mathbf{CrsMon}^{/\widetilde G}_{\widetilde{\mathcal A}}$ to the pullback $\varphi^\natural_{\widetilde G}$ is fully faithful.
\end{corollary}
\begin{proof}
Suppose $\varphi$ is fully faithful, so we may regard $\mathcal A$ as a full subcategory of $\widetilde{\mathcal A}$.
It suffices to show the unit $M\to\varphi^\natural_{\widetilde G}\varphi^{\widetilde G}_\flat M$ of the adjunction $\varphi^{\widetilde G}_\flat\dashv\varphi^\natural_{\widetilde G}$ is an isomorphism for each $M\in\mathbf{CrsMon}^{/\varphi^\ast\widetilde G}_{\mathcal A}$.
Note that, for $a\in\mathcal A$, it is given by
\begin{equation}
\label{eq:prf:Lanff-crsmon:unit}
M(a)\to\varphi^\natural_{\widetilde G}\varphi^{\widetilde G}_\flat M(a)=\varphi^{\widetilde G}_\flat M(a)
\ ;\quad x \mapsto [x,\mathrm{id}]\ .
\end{equation}
On the other hand, since $\mathcal A$ is a full subcategory of $\widetilde{\mathcal A}$, we have
\[
[y,\widetilde\varphi]=[\widetilde\varphi^\ast(y),\mathrm{id}]
\]
for each $y\in M(b)$ and $\widetilde\varphi\in\widetilde{\mathcal A}(a,b)=\mathcal A(a,b)$.
This implies that the map \eqref{eq:prf:Lanff-crsmon:unit} is surjective.
Moreover, the faithfulness of the left Kan extension $\varphi^{\widetilde G}_!:\mathbf{Set}^{/\varphi^\ast\widetilde G}_{\mathcal A}\to\mathbf{Set}^{\widetilde G}_{\widetilde{\mathcal A}}$ implies that $[x,\mathrm{id}]=[x',\mathrm{id}]$ if and only if $x=x'$.
Thus, \eqref{eq:prf:Lanff-crsmon:unit} is also injective, so we obtain the result.
\end{proof}

\begin{remark}
\Cref{cor:Lanff-crsmon} also follows from \cref{prop:Ranff-crsmon} and the fact that, for adjunctions $L\dashv F\dashv R$, $L$ is fully faithful if and only if so is $R$.
It seems to be a kind of \emph{falklore} while the reader will find proofs in \cite{DyckhoffTholen1987} and \cite{KellyLawvere1989}.
\end{remark}

\begin{example}
Take $\varphi$ to be the embedding $j:\Delta\to\widetilde\Delta$ and $\widetilde G=\mathfrak W_{\widetilde\Delta}$, then we obtain an adjunction
\begin{equation}
\label{eq:Lan-aug}
j_\flat:
\vcenter{
  \xymatrix{
    \mathbf{CrsMon}_\Delta \ar@{}[r]|-\perp \ar@/^.2pc/[]+R+(0,1);[r]+L+(0,1) & \mathbf{CrsMon}_{\widetilde\Delta} \ar@/^.2pc/[]+L+(0,-1);[l]+R+(0,-1) }}
:j^\natural
\end{equation}
by \cref{theo:Lan-crsmon}.
Since $j$ is fully faithful, by virtue of \cref{cor:Lanff-crsmon}, for every every crossed simplicial monoid $M_\bullet$, we have a canonical identification
\[
j_\flat M(\langle n\rangle) \cong M_{n-1}
\]
for each $n\ge 1$.
On the other hand, since the only object $\langle 0\rangle\in\widetilde\Delta$ outside the image of $j$ is initial, for the left Kan extension $j_!M$, we have
\[
j_!M(\langle 0\rangle)
\cong \colim_\Delta M_\bullet
\cong \operatorname{coeq}\left(M_1\xrightrightarrows{d_0,d_1}M_0\right)
\cong \pi_0(M)\ ,
\]
where $\pi_0(M)$ is the \emph{set of connected components} of the simplicial set $M$.
We claim $\pi_0(M)$ inherits a structure of monoids through the quotient map $M_0\to\pi_0(M)$.
Note that $\pi_0(M)$ is obtained as the quotient of $M_0$ by the equivalence relation $\sim$ generated by
\[
d_0 x\sim d_1 x
\]
for each $x\in M_1$.
Hence, to verify the claim, it suffices to find $z,z'\in M_1$ which support
\[
u\cdot d_0x\sim u\cdot d_1x
\ ,\quad
d_0x\cdot u\sim d_1 x\cdot u
\]
respectively for every $x\in M_1$ and $u\in M_0$.
Actually, we can take $z:=s_0u\cdot x$ and $z':=x\cdot s_0u$; indeed, we have
\[
\begin{gathered}
d_iz
= d_{x(i)}s_0u\cdot d_ix
= u\cdot d_ix\ ,
\\
d_iz'
= d_{s_0u(i)}x\cdot d_is_0u
= d_{s_0u(i)}x\cdot u\ .
\end{gathered}
\]
It turns out that the monoidal structure on $\pi_0M$ above makes not only the map $M_0\to\pi_0 M$ but also $M_n\to\pi_0M$ for arbitrary $n\in\mathbb N$ into a monoid homomorphism.
Finally, the left adjoint functor $j_\flat$ in \eqref{eq:Lan-aug} is given by
\[
j_\flat M(\langle n\rangle) \cong
\begin{cases}
\pi_0M \quad n=0 \\
M_{n-1}\quad n\ge 1\ ,
\end{cases}
\]
where $\pi_0M$ is equipped with the monoid structure given above.
\end{example}

\begin{example}
Take $\varphi$ to be the functor $\mathfrak J:\widetilde\Delta\to\nabla$.
Set $\mathcal I$ to be the set of morphisms $\varphi:\double\langle m\double\rangle\to\double\langle n\double\rangle\in\nabla$ such that it restricts to a bijection
\[
\varphi^{-1}\{1,\dots,n\}\to \{1,\dots,n\}\ .
\]
It is known that every morphism in $\nabla$ uniquely factors through a morphism in $\mathcal I$ followed by one in the image of $\mathfrak J$.
This factorization gives us a nice description of the left Kan extension functor
\[
\mathfrak J_!:\mathbf{Set}_{\widetilde\Delta}\to\mathbf{Set}_\nabla
\]
as follows: for an augmented simplicial set $X$, $\mathfrak J_!X(\double\langle n\double\rangle)$ is the set
\begin{equation}
\label{eq:Lan-interval:explicit}
\left\{(x,\rho)\ \middle|\ k\in\mathbb N,\ x\in X(k),\ \rho:\double\langle n\double\rangle\to\double\langle k\double\rangle\in\mathcal I\right\}
\end{equation}
For $\varphi:\double\langle m\double\rangle\to\double\langle n\double\rangle\in\nabla$, the induced map $\varphi^\ast:\mathfrak J_!X(\double\langle n\double\rangle)\to\mathfrak J_!X(\double\langle m\double\rangle)$ is given by
\[
\varphi^\ast(x,\rho) = (\mu^\ast(x),\rho_\varphi)\ ,
\]
where $(\mu,\rho_\varphi)$ is the unique pair of morphisms with $\mu\in\mathfrak J(\widetilde\Delta)$, $\rho_\varphi\in\mathcal I$, and $\mu\rho_\varphi=\rho\varphi$.
Thus, the left adjoint $\mathfrak J^{\mathfrak W_\nabla}_\flat$ in the adjunction
\begin{equation}
\label{eq:Lan-interval:adj}
\mathfrak J^{\mathfrak W_\nabla}_\flat:
\vcenter{
  \xymatrix{
    \mathbf{CrsMon}_{\widetilde\Delta}^{/\mathfrak W_{\widetilde\Delta}\times C_2} \ar@{}[r]|-\perp \ar@/^.2pc/[]+R+(0,1);[r]+L+(0,1) & \mathbf{CrsMon}_\nabla \ar@/^.2pc/[]+L+(0,-1);[l]+R+(0,-1) }}
:\mathfrak J^\natural_{\mathfrak W_\nabla}
\end{equation}
is described as follows: for $M\in\mathbf{CrsMon}_{\widetilde\Delta}^{/\mathfrak W_{\widetilde\Delta}\times C_2}$ with the associated augmented simplicial map $\theta:M\to C_2$, $\mathfrak J^{\mathfrak W_\nabla}_\flat M(\double\langle n\double\rangle)$ is the quotient of the free monoid over the set defined similarly to \eqref{eq:Lan-interval:explicit} by the congruence relation generated by
\begin{equation}
\label{eq:Lan-interval:rel}
(xy,\rho) \sim (x,\rho^{\theta(y)})(y,\rho)
\ ,\quad (e_k,\rho) \sim e_n\ .
\end{equation}
In particular, if $\theta:M\to C_2$ is trivial, the relation \eqref{eq:Lan-interval:rel} gives rise to an isomorphism
\bgroup
\[
\mathfrak J_\flat M(\double\langle n\double\rangle)
\cong \bigast_{\rho:\double\langle n\double\rangle\to\double\langle k\double\rangle\in\mathcal I} M(\langle k\rangle)\ ,
\]
\egroup
where the right hand side is the free product of monoids.
\end{example}

To end the section, we mention crossed groups.
We saw above that the Kan extensions along stable faithful functors give rise to adjunctions between the category of crossed monoids.
Notice that all the construction can be described as limits and colimits, at least degreewisely.
On the other hand, in view of \cref{theo:crsgrp-birefl}, the subcategory $\mathbf{CrsGrp}_{\mathcal A}\subset\mathbf{CrsMon}_{\mathcal A}$ is closed under both limits and colimits.
This implies all the discussion above restricts to crossed groups, so we obtain the following result.

\begin{theorem}
\label{theo:Kan-crsgrp}
Let $\varphi:\mathcal A\to\widetilde{\mathcal A}$ be a faithful functor, and let $\widetilde G$ be a crossed $\widetilde{\mathcal A}$-group such that $\varphi$ is $\widetilde G$-stable.
Then, the adjunctions
\[
\varphi^{\widetilde G}_\flat,\varphi^{\widetilde G}_\sharp:
\vcenter{
  \xymatrix{
    \mathbf{CrsMon}_{\mathcal A}^{/\varphi^\ast\widetilde G} \ar@/^.4pc/[]+R+(0,2);[r]+L+(0,2) \ar@/_.4pc/[]+R+(0,-2);[r]+L+(0,-2) & \mathbf{CrsMon}_{\widetilde{\mathcal A}}^{/\widetilde G} \ar[l]^-\perp_-\perp }}
:\varphi^\natural_{\widetilde G}
\]
given in \cref{cor:Ran-crsmon} and \cref{theo:Lan-crsmon} restrict to
\[
\varphi^{\widetilde G}_\flat,\varphi^{\widetilde G}_\sharp:
\vcenter{
  \xymatrix{
    \mathbf{CrsGrp}_{\mathcal A}^{/\varphi^\ast\widetilde G}  \ar@/^.4pc/[]+R+(0,2);[r]+L+(0,2) \ar@/_.4pc/[]+R+(0,-2);[r]+L+(0,-2) & \mathbf{CrsGrp}_{\widetilde{\mathcal A}}^{/\widetilde G} \ar[l]^-\perp_-\perp }}
:\varphi^\natural_{\widetilde G}\ .
\]
Moreover, if $\varphi$ is fully faithful, so are $\varphi^{\widetilde G}_\flat$ and $\varphi^{\widetilde G}_\sharp$ even after restricted to crossed groups.
\end{theorem}

\begin{appendices}
\section{Classification of crossed interval groups}
\label{sec:clsfy-crsint}

In this appendix, we give a complete list of crossed interval subgroups of the terminal crossed interval group $\mathfrak W_\nabla$.
As mentioned at the end of \cref{sec:terminal}, this gives us a classification of crossed interval groups up to \emph{non-crossed} parts (see \cref{prop:crsgrp-clsfy}).
Recall that, according to the computation in \cref{ex:crsgrp-intWeyl}, we have
\begin{equation}
\label{eq:intWeyl-split}
\mathfrak W_\nabla(\double\langle n\double\rangle)
\cong H_n\times C_2\ ,
\end{equation}
where $H_n$ is the $n$-th hyperoctahedral group, and $C_2$ is the group of order $2$.
At first glance, this looks pretty nice; we have an excellent theorem, namely, \emph{Goursat's Lemma} \cite{Goursat1889} to seek subgroups of a product of groups.
Unfortunately, the ``product'' in \eqref{eq:intWeyl-split} is, however, not actually the product of interval sets; namely, the second component $C_2$ is not really closed under the interval set structure so that the projections $H_n\times C_2\to H_n$ fails to define a map of interval sets.
This is because of the morphisms in $\nabla$ outside the image of $\widetilde\Delta$, so we first consider its restriction to $\widetilde\Delta$.
Indeed, let $\mathfrak j:\widetilde\Delta\to\nabla$ be the functor given by
\[
\mathfrak j(\langle n\rangle)
=\double\langle n\double\rangle
\]
(see \cref{ex:Ran-interval}).
Since $\mathfrak j$ is $\mathfrak W_\nabla$-stable faithful functor, in view of \cref{cor:Ran-crsmon}, it induces a functor
\[
\mathfrak j^\natural:\mathbf{CrsGrp}_\nabla\to\mathbf{CrsGrp}^{/\mathfrak j^\ast\mathfrak W_\nabla}_{\widetilde\Delta}\ .
\]
A good news is that \eqref{eq:intWeyl-split} now exhibits $\mathfrak j^\natural\mathfrak W_\nabla$ as a product $\mathfrak W_{\widetilde\Delta}\times C_2$ of augmented simplicial sets, where $C_2$ is the constant augmented simplicial set at $C_2$.
Note that $\mathfrak j^\natural$ is faithful so it preserves monomorphisms.
Hence every crossed interval subgroups of $\mathfrak W_\nabla$ is sent to an augmented crossed simplicial subgroup of $\mathfrak j^\natural\mathfrak W_\nabla\cong\mathfrak W_{\widetilde\Delta}\times C_2$.

To compute all the augmented crossed simplicial subgroups of $\mathfrak W_{\widetilde\Delta}\times C_2$, we establish a crossed analogue of Goursat's Lemma.
This can actually be done for general base categories $\mathcal A$.

\begin{lemma}
\label{lem:crsgrp-coker}
Let $\mathcal A$ be a small category.
Suppose we are given an inclusion $N\hookrightarrow G$ of crossed $\mathcal A$-groups such that $N(a)\subset G(a)$ is a normal subgroup for each $a\in\mathcal A$.
Then, the family $\{G(a)/N(a)\}_a$ admits a unique structure of $\mathcal A$-sets such that the canonical map $G(a)\twoheadrightarrow G(a)/N(a)$ is a map of $\mathcal A$-sets.
\end{lemma}
\begin{proof}
The uniqueness follows from the surjectivity of each map $G(a)\to G(a)/N(a)$.
We show that, for each $\varphi:b\to a\in\mathcal A$, the map
\begin{equation}
\label{eq:prf:crsgrp-coker:def}
\varphi^\ast:G(a)/N(a)\to G(b)/N(b)
\ ;\quad x N(a)\to \varphi^\ast(x) N(b)
\end{equation}
is well-defined.
Note that, since $N(a)$ and $N(b)$ are normal, it is equivalent to see the same statement for right cosets.
For each $x\in G(a)$ and for every $u\in N(a)$, we have
\[
\varphi^\ast(ux)
= (\varphi^x)^\ast(u)\varphi^\ast(x)\ .
\]
Since $N$ is a crossed $\mathcal A$-subgroup of $G$, $(\varphi^x)^\ast(u)\in N(b)$ so we obtain $\varphi^\ast(ux)\in N(b)\varphi^\ast(x)$.
This immediately implies \eqref{eq:prf:crsgrp-coker:def} in fact defines an $\mathcal A$-set structure on the family $\{G(a)/N(a)\}_a$.
The required property is easily verified.
\end{proof}

In what follows, we write $G/N$ the $\mathcal A$-set obtained in \cref{lem:crsgrp-coker}.
Notice that it admits a canonical degreewise group structure induced from $G$.

Similarly to the ordinary Grousat's Lemma, we aim to present subgroups of a product of crossed $\mathcal A$-groups in terms of subgroups of each components.
Here, the term ``product'' is ambiguous; indeed, the \emph{cartesian product} in the category $\mathbf{CrsGrp}_{\mathcal A}$ does not always agree with the product of $\mathcal A$-sets, while the latter does not always produce crossed $\mathcal A$-groups even if made from crossed $\mathcal A$-groups.
For example, our target $\mathfrak W_{\widetilde\Delta}\times C_2$ is not a cartesian product in the category $\mathbf{CrsGrp}_{\widetilde\Delta}$.
Hence, we need to find an appropriate notion to substitute for \emph{products}.
A key observation is that, for a group $G$, to establish an isomorphism $G\cong G^{(1)}\times G^{(2)}$, it suffices to find a pair $(G^{(1)},G^{(2)})$ of subgroups of $G$ such that $G$ is generated by $G^{(1)}\cup G^{(2)}$ and
\[
G^{(1)}\cap G^{(2)}=[G^{(1)},G^{(2)}]=\{e\}\ ,
\]
where the middle is the commutator subgroup.

\begin{definition}
Let $\mathcal A$ be a small category.
A crossed $\mathcal A$-group $G$ is said to be a \emph{virtual product} of crossed $\mathcal A$-subgroups $G^{(1)}$ and $G^{(2)}$ if the following conditions hold:
\begin{enumerate}[label={\rm(\roman*)}]
  \item the map $G^{(1)}\ast G^{(2)}\to G$ induced by the inclusions is an epimorphism in $\mathbf{CrsGrp}_{\mathcal A}$, where $G^{(1)}\ast G^{(2)}$ is the free product (see \cref{prop:crsgrp-colim});
  \item the pullback $G^{(1)}\times_G G^{(2)}$ is trivial; roughly, we often write $G^{(1)}\cap G^{(2)}=\ast$;
  \item for each $a\in\mathcal A$, the commutator subgroup $[G^{(1)}(a),G^{(2)}(a)]\subset G(a)$ is trivial; in other words, elements of $G^{(1)}$ and $G^{(2)}$ commute with each other.
\end{enumerate}
\end{definition}

\begin{lemma}
\label{lem:virtprod-inv}
Let $\mathcal A$ be a small category, and let $G$ be a crossed $\mathcal A$-group which is a virtual product of crossed $\mathcal A$-subgroups $G^{(1)}$ and $G^{(2)}$.
Then, for every morphism $\varphi:a\to b$, and for each element $x_i\in G^{(i)}$ for $i=1,2$, we have
\[
(\varphi^{x_2})^\ast(x_1)=\varphi^\ast(x_1)
\ ,\quad
(\varphi^{x_1})^\ast(x_2)=\varphi^\ast(x_2)\ .
\]
\end{lemma}
\begin{proof}
The condition on crossed groups implies
\[
\varphi^\ast(x_1 x_2) = (\varphi^{x_2})^\ast(x_1)\varphi^\ast(x_2)
\ ,\quad
\varphi^\ast(x_2 x_1) = (\varphi^{x_1})^\ast(x_2)\varphi^\ast(x_1)\ .
\]
Since the commutator subgroups $[G^{(1)}(a),G^{(2)}(a)]\subset G(a)$ and $[G^{(1)}(b),G^{(2)}(b)]\subset G(b)$ are trivial, the elements above equal, and we obtain
\[
\varphi^\ast(x_1)^{-1}(\varphi^{x_2})^\ast(x_1)
= (\varphi^{x_1})^\ast(x_2)\varphi^\ast(x_2)^{-1}\ .
\]
The left hand side belongs to $G^{(1)}(b)$ while the right to $G^{(2)}(b)$, so both belong to $G^{(1)}(b)\cap G^{(2)}(b)=\{e\}$.
Thus, the result follows.
\end{proof}

\begin{theorem}[Goursat's Lemma for crossed groups]
\label{theo:Goursat}
Let $\mathcal A$ be a small category.
Suppose $G$ is a crossed $\mathcal A$-group which is a virtual product of crossed $\mathcal A$-subgroups $G^{(1)},G^{(2)}\subset G$.
Then, there is a $1$-$1$ correspondence between the following data:
\begin{enumerate}[label={\rm(\alph*)}]
  \item a crossed $\mathcal A$-subgroup $H$ of $G$;
  \item\label{data:Goursat:quintuple} a quintuple $(\widetilde H^{(1)},H^{(1)};\widetilde H^{(2)},H^{(2)};\chi)$ of
\begin{enumerate}[label={\rm(\roman*)}]
  \item crossed $\mathcal A$-subgroups $H^{(i)}\subset\widetilde H^{(i)}\subset G^{(i)}$ for $i=1,2$ so that $H^{(i)}(a)$ is a normal subgroup in $\widetilde H^{(i)}(a)$ for each $a\in\mathcal A$;
  \item a map $\chi:\widetilde H^{(1)}/H^{(1)}\to\widetilde H^{(2)}/H^{(2)}$ of $\mathcal A$-sets which is a degreewise group isomorphism.
\end{enumerate}
\end{enumerate}
\end{theorem}
\begin{proof}
We denote by $\operatorname{Sub}(G)$ the set of crossed $\mathcal A$-subgroups of $G$ and by $\operatorname{Gou}(G^{(1)},G^{(2)})$ the set of quintuples as in \ref{data:Goursat:quintuple}.
For $Q=(\widetilde H^{(1)},H^{(1)};\widetilde H^{(2)},H^{(2)};\chi)\in\operatorname{Gou}(G^{(1)},G^{(2)})$, consider the subset
\[
H^Q(a)
:=\left\{x_1x_2\ \middle|\ x_1\in\widetilde H^{(1)}(a),\ x_2\in\widetilde H^{(2)}(a),\ \chi\left(x_1 H^{(1)}(a)\right)= x_2 H^{(2)}(a)\right\}
\]
of $G(a)$ for each $a\in\mathcal A$.
We see the family $H^Q=\{H^Q(a)\}_a$ forms a crossed $\mathcal A$-subgroup of $G$.
Since $[G^{(1)}(a),G^{(2)}(a)]=\ast$ and $\chi$ is a degreewise group isomorphism, $H^Q$ is a degreewise subgroup of $G$.
On the other hand, for $\varphi:b\to a\in\mathcal A$, \cref{lem:virtprod-inv} implies that, for each $x_1x_2\in H^Q(a)$ with $x_i\in\widetilde H^{(i)}(a)$ and $\chi\left(x_1H^{(1)}(a)\right)=x_2 H^{(2)}(a)$,
\[
\varphi^\ast(x_1x_2)
= (\varphi^{x_2})^\ast(x_1)\varphi^\ast(x_2)
= \varphi^\ast(x_1)\varphi^\ast(x_2)\ .
\]
Since $\chi$ is an $\mathcal A$-map, we have
\[
\chi\left(\varphi^\ast(x_1) H^{(1)}(b)\right)
= \varphi^\ast\chi\left(x_1 H^{(1)}(a)\right)
= \varphi^\ast(x_2)H^{(2)}(b)
\]
so $\varphi^\ast(x_1x_2)\in H^Q(b)$.
Hence, $H^Q\subset G$ is closed under both the degreewise group structure and the $\mathcal A$-set structure so to define a crossed $\mathcal A$-subgroup.

Now, consider the map
\begin{equation}
\label{eq:prf:Goursat:corr}
\operatorname{Gou}(G^{(1)},G^{(2)})\to\operatorname{Sub}(G)
\ ;\quad Q\mapsto H^Q\ .
\end{equation}
We show it admits an inverse.
Suppose $H\subset G$ is a crossed $\mathcal A$-subgroup.
For $\{i,j\}=\{1,2\}$, and for $a\in\mathcal A$, we define
\[
\begin{gathered}
\widetilde H^{(i)}(a)
:=\left\{x_i\in G^{(i)}(a)\ \middle|\ \exists x_j\in G^{(j)}(a):x_ix_j\in H(a)\right\}
\\
H^{(i)}(a)=G^{(i)}(a)\cap H(a)\ .
\end{gathered}
\]
Note that the family $\widetilde H^{(i)}=\{\widetilde H^{(i)}(a)\}_a$ forms a crossed $\mathcal A$-subgroup of $G^{(i)}$ as well as $H^{(i)}=\{H^{(i)}(a)\}_a$.
Indeed, if $(x_1,x_2),(x'_1,x'_2)\in G^{(1)}(a)\times G^{(2)}(a)$ are pairs with $x_1x_2,x'_1x'_2\in H(a)$, and if $\varphi:b\to a\in\mathcal A$, then by \cref{lem:virtprod-inv},
\begin{equation}
\label{eq:prf:Goursat:A-hom}
\begin{gathered}
(x_1x'_1)(x_2x'_2)
= (x_1x_2)(x'_1x'_2)
\in H(a)\ ,
\\
\varphi^\ast(x_1)\varphi^\ast(x_2)
= \varphi^\ast(x_1x_2)
\in H(b)\ .
\end{gathered}
\end{equation}
On the other hand, $H^{(i)}(a)$ is clearly a normal subgroup of $H(a)$.
Moreover, for $x_i\in\widetilde H^{(i)}(a)$, take an element $x_j\in G^{(j)}(a)$ with $x_ix_j\in H(a)$, then
\[
x_i H^{(i)}(a)
= x_i x_j H^{(i)}(a) x_j^{-1}
= (x_ix_j H^{(i)}(a) (x_ix_j)^{-1}) x_i
= H^{(i)}(a) x_i\ ,
\]
which implies $H^{(i)}(a)$ is a normal subgroup of $\widetilde H^{(i)}(a)$.
Hence, by \cref{lem:crsgrp-coker}, we obtain two $\mathcal A$-sets $\widetilde H^{(1)}/H^{(1)}$ and $\widetilde H^{(2)}/H^{(2)}$ that are degreewise groups.
Define a map $\chi^H_a:(\widetilde H^{(1)}/H^{(1)})(a)\to(\widetilde H^{(2)}/H^{(2)})(a)$ so that
\[
\chi^H_a\left(x_1 H^{(1)}(a)\right)=x_2 H^{(2)}(a)
\]
if and only if $x_1x_2\in H(a)$.
It is easily verified that such a map $\chi_a$ is uniquely determined by the crossed $\mathcal A$-subgroup $H\subset G$.
Furthermore, the formulae \eqref{eq:prf:Goursat:A-hom} implies that $\chi^H=\{\chi^H_a\}_a$ defines an $\mathcal A$-map $\widetilde H^{(1)}/H^{(1)}\to\widetilde H^{(2)}/H^{(2)}$ that is a degreewise group isomorphism.
We write
\[
Q^H:=(\widetilde H^{(1)},H^{(1)};\widetilde H^{(2)},H^{(2)};\chi^H)
\]
the resulting quintuple.
Then, clearly $Q^H\in\operatorname{Gou}(G^{(1)},G^{(2)})$, and the classical Goursat's Lemma for groups shows that the assignment $H\mapsto Q^H$ gives the inverse of the map \eqref{eq:prf:Goursat:corr}.
\end{proof}

In the rest, we compute all the crossed interval subgroups of $\mathfrak W_\nabla$.
To begin with, we focus on the crossed interval subgroup $\mathfrak H\subset\mathfrak W_\nabla$ of hyperoctahedral groups whose structure is given in \cref{ex:crsgrp-intHypoct}.
Since we have $\mathfrak j^\ast\mathfrak H\cong\mathfrak W_{\widetilde\Delta}$, crossed interval subgroups of $\mathfrak H$ are augmented crossed simplicial subgroups of $\mathcal W_{\widetilde\Delta}$ closed under the structure of interval sets.
As a result of \cite{FL91}, we have a complete list of crossed simplicial subgroups of $\mathfrak W_\Delta$ as in \cref{tb:simpcrs-clsfy} in \cref{ex:simpcrs-clsfy}.
In view of \cref{ex:Ran-aug}, to obtain a complete list of augmented crossed simplicial subgroups of $\mathfrak W_{\widetilde\Delta}$, we only have to shift the list in simplicial case by $1$.
Hence, the result is indicated in \cref{tb:augsimp-clsfy}:
\begin{table}[htbp]
\centering
\begin{tabular}{c|c|c}
  \textbf{name} & \textbf{symbol} & \textbf{group at $\langle n\rangle$} \\\hline
  Trivial & $*$ & $1$ \\
  Reflexive & $C_2$ & $C_2$ \\
  Cyclic & $\Lambda$ & $C_n$ \\
  Dihedral & $\mathfrak D$ & $D_n$ \\
  Symmetric & $\mathfrak S$ & $\mathfrak S_n$ \\
  Reflexosymmetric & $\widetilde{\mathfrak S}$ & $\mathfrak S_n\times C_2$ \\
  Weyl (Hyperoctahedral) & $\mathfrak W_{\widetilde\Delta}\cong\mathfrak H$ & $H_n$
\end{tabular}
\caption{The augmented crossed simplicial subgroups of $\mathfrak W_{\widetilde\Delta}$}
\label{tb:augsimp-clsfy}
\end{table}
It is seen that, among \cref{tb:augsimp-clsfy}, crossed \emph{interval} subgroups are precisely the trivial one $\ast$, the symmetric one $\mathfrak S$, and the hyperoctahedral one $\mathfrak H$ itself.

Now suppose $H\subset\mathfrak W_\nabla$ is a crossed interval subgroup.
Since the augmented crossed simplicial group $\mathfrak j^\ast\mathfrak W_\nabla=\mathfrak W_{\widetilde\Delta}\times C_2$ is, as suggested by the notation, a virtual product of $\mathfrak W_{\widetilde\Delta}\cong\mathfrak H$ and $C_2$, \cref{theo:Goursat} implies there is an associated quintuple $(\widetilde H^{(1)},H^{(1)};\widetilde H^{(2)},H^{(2)};\chi^H)$ of augmented crossed simplicial subgroups $H^{(1)}\subset\widetilde H^{(1)}\subset\mathfrak H$, $H^{(2)}\subset\widetilde H^{(2)}\subset C_2$, and $\chi:\widetilde H^{(1)}/H^{(1)}\cong\widetilde H^{(2)}/H^{(2)}$.
In particular, according to the proof of \cref{theo:Goursat}, $H^{(1)}=H\cap\mathfrak H$ is an intersection of crossed interval subgroups of $\mathfrak W_\nabla$, so $H^{(1)}$ is a crossed interval subgroup of $\mathfrak H$, which is either $\ast$, $\mathfrak S$, or $\mathfrak H$ by the argument above.
On the other hand, since $\widetilde H^{(2)}/H^{(2)}$ is a subquotient of the group $C_2$, it is, degreewisely, of order at most $2$.
Thus, the isomorphism $\chi^H$ is, if exists, uniquely determined by the other data $(\widetilde H^{(1)},H^{(1)};\widetilde H^{(2)},H^{(2)})$, so we can omit it in what follows.
As a result, all the possibilities of the quadruple are listed below:
\[
\begin{multlined}
(\ast,\ast;\ast,\ast)\ ,
(\mathfrak S,\mathfrak S;\ast,\ast)\ ,
(\mathfrak H,\mathfrak H;\ast,\ast)\ ,
(\ast,C_2;\ast,C_2)\ ,\\
(\mathfrak S,\widetilde{\mathfrak S};\ast,C_2)\ ,
(\ast,\ast;C_2,C_2)\ ,
(\mathfrak S,\mathfrak S;C_2,C_2)\ ,
(\mathfrak H,\mathfrak H;C_2,C_2)\ .
\end{multlined}
\]
It turns out that the sixth and seventh do not produce crossed interval subgroups while the others do.
Hence, we finally obtain the list of crossed interval subgroups of $\mathfrak W_\nabla$ (\cref{tb:crsint-clsfy}).
\begin{table}[htbp]
\centering
\begin{tabular}{c|c|c|c}
  \textbf{name} & \textbf{symbol} & \textbf{group at $\double\langle n\double\rangle$} & $\substack{\textbf{associated}\cr\textbf{quadruple}}$ \\\hline
  Trivial & $*$ & $1$ & $(\ast,\ast;\ast,\ast)$ \\
  Reflexive & $C_2$ & $C_2$ & $(\ast,C_2;\ast,C_2)$ \\
  Symmetric & $\mathfrak S$ & $\mathfrak S_n$ & $(\mathfrak S,\mathfrak S;\ast,\ast)$ \\
  Reflexosymmetric & $\widetilde{\mathfrak S}$ & $\mathfrak S_n\times C_2$ & $(\mathfrak S,\widetilde{\mathfrak S};\ast,C_2)$ \\
  Hyperoctahedral & $\mathfrak H$ & $H_n$ & $(\mathfrak H,\mathfrak H;\ast,\ast)$ \\
  Weyl & $\mathfrak W_\nabla$ & $H_n\times C_2$ & $(\mathfrak H,\mathfrak H;C_2,C_2)$
\end{tabular}
\caption{The crossed interval subgroups of $\mathfrak W_\nabla$}
\label{tb:crsint-clsfy}
\end{table}
\end{appendices}

\bibliographystyle{plain}
\bibliography{mybiblio}
\end{document}